\documentclass{amsart}

\usepackage{amsmath}
\usepackage{amssymb}
\usepackage{mathrsfs}
\usepackage{enumerate}
\usepackage{color}

\newtheorem{theorem}{Theorem}[section]
\newtheorem{lemma}[theorem]{Lemma}
\newtheorem{proposition}[theorem]{Proposition}
\newtheorem{definition}[theorem]{Definition}
\newtheorem{corollary}[theorem]{Corollary}
\newtheorem{hypothesis}[theorem]{Hypothesis}
\newtheorem{remark}[theorem]{Remark}

\DeclareMathOperator{\const}{const\,}

\newcommand{\bbC}{{\mathbb{C}}}

\newcommand{\bbN}{{\mathbb{N}}}

\newcommand{\bbR}{{\mathbb{R}}}

\newcommand{\bbZ}{{\mathbb{Z}}}

\newcommand{\ba}{{\mathbf{a}}}

\newcommand{\bt}{{\mathbf{t}}}
\newcommand{\bs}{{\mathbf{s}}}
\newcommand{\be}{{\mathbf{e}}}
\newcommand{\bu}{{\mathbf{u}}}

\newcommand{\bx}{{\mathbf{x}}}
\newcommand{\bn}{{\mathbf{n}}}
\newcommand{\bm}{{\mathbf{m}}}
\newcommand{\bk}{{\mathbf{k}}}

\newcommand{\cA}{{\mathcal A}}
\newcommand{\cB}{{\mathcal B}}

\newcommand{\cD}{{\mathcal D}}
\newcommand{\cE}{{\mathcal E}}
\newcommand{\cF}{{\mathcal F}}

\newcommand{\cH}{{\mathcal H}}

\newcommand{\cL}{{\mathcal L}}
\newcommand{\cM}{{\mathcal M}}
\newcommand{\cN}{{\mathcal N}}

\newcommand{\cS}{{\mathcal S}}

\newcommand{\tr}{\operatorname{tr}}
\newcommand{\dom}{\operatorname{dom}}
\newcommand{\sgn}{\operatorname{sgn}}

\begin{document}

\date{\today}
\title{Cwikel estimates revisited}
\author[G.\ Levitina]{Galina Levitina} 
\address{School of Mathematics and Statistics, UNSW, Kensington, NSW 2052,
Australia} 
\email{g.levitina@unsw.edu.au}

\author[F.\ Sukochev]{Fedor Sukochev}
\address{School of Mathematics and Statistics, UNSW, Kensington, NSW 2052,
Australia} 
\email{f.sukochev@unsw.edu.au}

\author[D.\ Zanin]{Dmitriy Zanin}
\address{School of Mathematics and Statistics, UNSW, Kensington, NSW 2052,
Australia} 
\email{d.zanin@unsw.edu.au}
\maketitle

\section{Introduction}

One of the most beautiful results in operator theory are the so-called Cwikel estimates concerning the singular values of the operator $M_fg(-i\nabla)$ on $L_2(\bbR^d)$, where $M_f$ denotes the multiplication operator on $L_2(\bbR^d)$ by the function $f$ and $\nabla$ denotes the gradient \cite{Simon}. It was conjectured by Simon \cite{Simon_conjecture} and proved by Cwikel \cite{Cwikel} (see also \cite{BS}) that conditions $f\in L_p(\bbR^d)$ and $g\in L_{p,\infty}(\bbR^d)$, $p>2$, ensure that the operator $M_fg(-i\nabla)$ belongs to the weak Schatten ideal $\cL_{p,\infty}(L_2(\bbR^d))$ and the norm of $M_fg(-i\nabla)$ in $\cL_{p,\infty}(L_2(\bbR^d))$ is dominated by the product of $\|f\|_p$ and $\|g\|_{p,\infty}$. Estimates given in Simon's book \cite{Simon} and deeper results in this direction can be found in \cite{BKS},  \cite{BS}, \cite{Frank},  \cite{Hundertmark}, \cite{Weidl}. 

The operators $M_f g(-i\nabla)$ and $f(i\nabla)M_g$ are unitarily equivalent, and therefore, their singular numbers coincide. However, the results given in \cite{Cwikel,Simon} do not reflect the fact that the functions $f$ and $g$ are interchangeable. Noting this deficiency, T.Weidl in \cite{Weidl} and \cite{Weidl_via_Cesaro} suggested a new, symmetric, approach to Cwikel estimates, proving that 
$$\mu^2_{\cB(L_2(\bbR^d))}(M_fg(-i\nabla))\leq\const C\mu^2_{L_\infty(\bbR^{2d})}(f\otimes g),\quad f\otimes g\in (L_2+L_\infty)(\bbR^{2d})$$
with some absolute constant. Here, $\mu(x)$ denotes the singular value function of an operator $x$ (see Section \ref{prel} for definition) and $C$ is the Cesaro operator. 
In this paper we prove a much stronger version of this result, namely that
\begin{equation}\label{intro_Cesaro}
C\mu^2(M_fg(-i\nabla))\leq\const  C\mu^2(f\otimes g).
\end{equation}
The trivial estimate $\mu^2(M_fg(-i\nabla))\leq C\mu^2(M_fg(-i\nabla))$, immediately implies Weidl's result above. 

In fact, we prove a radically more general version of this estimate for noncommutative variables $f$ and $g$. Namely, let $\cA_1$ and $\cA_2$ be semifinite von Neumann algebras represented on the same (separable) Hilbert space $\cH$ and let $\tau_1$ and $\tau_2$ be faithful normal semifinite traces on $\cA_1$ and $\cA_2,$ respectively. We denote by $\pi_1$ (respectively, $\pi_2$) the representation of $\cA_1$ (respectively, $\cA_2$) on $\cH$. We assume that the representation of the algebras $\cA_1$ and $\cA_2$ are intertwined in such a way that the estimate holds in the Hilbert-Schmidt ideal,  that is, if $x\in \cL_2(\cA_1,\tau_1), y\in \cL_2(\cA_2,\tau_2)$, then $\pi_1(x)\pi_2(y)\in \cL_2(\cH)$ and 
$$\|\pi_1(x)\pi_2(y)\|_2\leq \const \|x\|_2\|y\|_2.$$ Under this assumption we prove our first main result, which implies \eqref{intro_Cesaro} in the special case when $\cA_1$ and $\cA_2$ coincide with $L_\infty(\bbR^d)$. 
\begin{theorem}\label{intro_thm1}
 Suppose that  $x\otimes y \in (\cL_2+\cL_{\infty})(\cA_1\otimes \cA_2,\tau_1\otimes \tau_2)$. Then $\pi_1(x)\pi_2(y)\in \cB(\cH)$ and
$$\mu_{\cB(\cH)}^2(\pi_1(x)\pi_2(y))\prec\prec 532\mu_{\cA_1\otimes \cA_2}^2(x\otimes y),$$
in the sense of Hardy-Littlewood majorisation.
\end{theorem}

As an immediate corollary of this estimate and the Lorentz-Shimogaki interpolation theorem \cite{LS_interpolation} we obtain abstract Cwikel estimates in an arbitrary symmetric space, which is an interpolation space for the Banach couple $(L_2,L_\infty)$. 

\begin{corollary}\label{intro_cor}
 Let $(E(0,\infty),\|\cdot\|_E)$ be an interpolation space for $(L_2,L_\infty)$ and let $x\in S(\cA_1,\tau_1), y\in S(\cA_2,\tau_2)$ be such that $x\otimes y\in E(\cA_1\otimes \cA_2)$. Then $\pi_1(x)\pi_2(y)\in  E(\cH)$ and 
$$\|\pi_1(x)\pi_2(y)\|_{E(\cH)}\leq C_E \|x\otimes y\|_{E(\mathcal{A}_1\otimes \mathcal{A}_2)}.$$
\end{corollary}
The result of Corollary \ref{intro_cor} strengthens results of \cite{Cwikel}, \cite[Section 4]{BKS}, \cite{Weidl}.
Furthermore, in Theorem \ref{Cwikel_L2,infty_contexample} we show that our assumption on the symmetric function space $E(0,\infty)$ in Corollary \ref{intro_cor} is optimal. That is, if we omit the assumption that $E(0,\infty)$ is an $(L_2,L_\infty)$-interpolation space, then the corresponding Cwikel estimates fail. The counterexample  (see Theorem \ref{Cwikel_L2,infty_contexample}) is yielded by the space  $L_{2,\infty}(0,\infty)$, which is not an $(L_2,L_\infty)$-interpolation space (and is not covered by \cite{BS},\cite{BKS},\cite{Cwikel},\cite{Simon},\cite{Weidl}) and uses the classical operator $M_fg(-i\nabla)$ in the one-dimensional setting. Nevertheless, we are in a position to present a new positive result for the ideal $\cL_{2,\infty}(L_2(\bbR^d))$ in the classical setting of operators $M_fg(-i\nabla)$.  Our approach here is based on proving first that the space $\cL_{1,\infty}(L_2(\bbR^d))$ is logarithmically convex (thus extending the classical result of \cite{Stein_Weiss_logconvexity} and \cite{Kalton_logconvex} to the noncommutative setting). Using this result we are able to prove a modified version of Cwikel estimates for the weak Schatten ideal $\cL_{2,\infty}(L_2(\bbR^d))$.
\begin{theorem}\label{intro_weak_L2}
If $g\in \ell_{2,\infty}(L_4)(\bbR^d)$ and 
$f\in \ell_{2,\log}(L_\infty)(\bbR^d)$, then $M_fg(-i\nabla)\in \cL_{2,\infty}(L_2(\bbR^d))$ and 
$$\|M_fg(-i\nabla)\|_{2,\infty}\leq \const \|f\|_{2,\log} \|g\|_{\ell_{2,\infty}(L_4)}.$$
\end{theorem}
(For the definition of the space $\ell_{2,\log}(L_\infty)(\bbR^d)$ the reader is referred to Definition \ref{def_l_2log}).

Now, we briefly describe our results for the case $0<p<2$, which are dramatically different from the case $p>2$ (see Corollary \ref{intro_cor}).  Recall that the assumption  $f,g\in L_1(\bbR^d)$ does not guarantee that the operator $M_fg(-i\nabla)$ belongs to the trace-class. The best result to date for the weak Schatten ideal $\cL_{p,\infty}(L_2(\bbR^d))$, $0<p<2$ (see \cite{BKS}) is that  $f\in \ell_p(L_2)(\bbR^d)$, $g\in \ell_{p,\infty}(L_2)(\bbR^d)$ guarantee that $M_fg(-i\nabla)\in \cL_{p,\infty}(L_2(\bbR^d))$. Observe that this result has the same asymmetry as Cwikel's result for $p>2$. We rectify this deficiency in the following 
\begin{theorem}\label{intro_p<2}
Let $0<p<2$.
\begin{enumerate}
\item  If $f\otimes g\in \ell_p(L_2)(\bbR^{2d})$, then $M_fg(-i\nabla)\in \cL_p(L_2(\bbR^d))$ and
$$\|M_fg(-i\nabla)\|_p\leq C_p \|f\otimes g\|_{\ell_{p}(L_2)(\bbR^{2d})}.$$
\item If $f\otimes g\in \ell_{p,\infty}(L_2)(\bbR^{2d})$, then $M_fg(-i\nabla)\in \cL_{p,\infty}(L_2(\bbR^d))$ and
$$\|M_fg(-i\nabla)\|_{p,\infty}\leq C_p \|f\otimes g\|_{\ell_{p,\infty}(L_2)(\bbR^{2d})}.$$
\end{enumerate} 
\end{theorem}
It should be pointed out that our approach of the proof of Theorem \ref{intro_p<2} is totally different and substantially simpler than those used in \cite{BS}, \cite{BKS}, \cite{Simon}. It is based on the classical (strong) majorisation introduced yet by Hardy and Littlewood,
 and allows us to substantially strengthen all earlier results.
 
 Now, we discuss applications of our results to classical problems in mathematical physics. Recall that Cwikel estimates were firstly applied in \cite{Cwikel} to prove the now famous Cwikel-Lieb-Rozenblum bounds on the number of negative eigenvalues of the Schroedinger operator $-\Delta+V$. Prompted by development of noncommutative geometry, in \cite{GG-BISV,Gayral_Iochum, Gayral_Iochum_Varilly} the notions of noncommutative plane (or rather noncommutative Euclidean space) and of the corresponding Dirac operator were introduced. In Section \ref{sec_nc_plane} we suggest a new approach to this notion, which is based on twisted convolution product approach. Our noncommutative Euclidean space $L_\infty(\bbR_\theta^d)$, where $\theta$ is $d\times d$ antisymmetric and real matrix, $d$ is even, is unitary equivalent to the Moyal plane \cite{Rieffel},\cite{Gayral_Iochum_Varilly}, however, our approach appears to be substantially simpler than the treatment based on the Moyal product. 
 The abstract Cwikel estimate established in Theorem \ref{intro_thm1} can be immediately applied to obtain the noncommutative analogue of  Cwikel-Lieb-Rozenblum bound (see Corollary \ref{nc_plane_CLR}) for the noncommutative Schroedinger operator $-\Delta_\theta+x$, where $\Delta_\theta$ is the Laplacian associated with $L_\infty(\bbR_\theta^d)$, $d\geq 4$, and $x\in L_\infty(\bbR_\theta^d)$ has certain decay properties.
 
We also note, in passing, that although our approach to the definition of Dirac operator associated to $L_\infty(\bbR_\theta^d)$ differes from \cite{GG-BISV}, it allows us to introduce Sobolev spaces $W^{m,p}(\bbR_\theta^d)$ in  $L_\infty(\bbR_\theta^d)$ in a natural way, mimicking the definition of the classical Sobolev spaces $W^{m,p}(\bbR^d)$. These Sobolev spaces $W^{m,p}(\bbR_\theta^d)$ are drastically different from the spaces introduced in \cite{Gayral_Iochum_Varilly} (see also \cite{CGRS_Memoirs}). Using this new definition of Sobolev spaces we can significantly improve Cwikel  estimates for weak Schatten ideals when $1\leq p\leq 2$  compared to those given in \cite{Gayral_Iochum_Varilly}. For $p=1$ these Cwikel estimate are crucial for summability condition for nonunital spectral triples (see \cite{CGRS_Memoirs}). 

Finally, in Section \ref{sec_magnetic} we discuss connection between the noncommutative Euclidean  space and the magnetic Laplacian and demonstrate the applicability of noncommutative Cwikel estimates to the latter object.

Throughout the paper we use
the convention that constants $C_{d}$, $c_{d,\theta}$, $\const$ etc
are strictly positive constants whose value depends only on their subscripts and can change from line to line.
\section{Preliminaries}\label{prel}

In what follows,  $\cH$ is a Hilbert space and $\cB(\cH)$ is the
$*$-algebra of all bounded linear operators on $\cH$, $\mathcal{M}$ is 
a von Neumann algebra on $\cH$, equipped with a faithful normal semifinite trace $\tau$. For details on von Neumann algebra
theory, the reader is referred to e.g. \cite{Dixmier}, \cite{KRI, KRII}
or \cite{Takesaki1}. 

A linear operator $x:\mathfrak{D}\left( x\right) \rightarrow \cH $,
where the domain $\mathfrak{D}\left( x\right) $ of $x$ is a linear
subspace of $\cH$, is said to be {\it affiliated} with $\mathcal{M}$
if $yx\subseteq xy$ for all $y\in \mathcal{M}^{\prime }$, where $\mathcal{M}^{\prime }$ is the commutant of $\mathcal{M}$. 

Let $x$ be a self-adjoint operator affiliated with $\mathcal{M}$.
We denote its spectral measure by $\{E_x\}$. It is well known that if
$x$ is a closed operator affiliated with $\mathcal{M}$ with the
polar decomposition $x = u|x|$, then $u\in\mathcal{M}$ and $E\in
\mathcal{M}$ for all projections $E\in \{E_{|x|}\}$.
An operator $x\in S\left( \mathcal{M}\right) $ is called $\tau$-measurable if there exists $\lambda>0$ such that $\tau(E_{|x|}(\lambda,
\infty))<\infty$. 

Consider the algebra $\mathcal{M}=L^\infty(0,\infty)$ of all
Lebesgue measurable essentially bounded functions on $(0,\infty)$.
Algebra $\mathcal{M}$ can be seen as an abelian von Neumann
algebra acting via multiplication on the Hilbert space
$\mathcal{H}=L^2(0,\infty)$, with the trace given by integration
with respect to Lebesgue measure $m.$
It is easy to see that the
algebra of all $\tau$-measurable operators
affiliated with $\mathcal{M}$ can be identified with
the subalgebra $S(0,\infty)$ of the algebra of Lebesgue measurable functions $L_0(0,\infty)$ which consists of all functions $x$ on $(0,\infty)$ such that
$m(\{|x|>s\})$ is finite for some $s>0$.

If $\mathcal{M}=\cB(\cH)$ (respectively, $l_\infty$) and $\tau$ is the
standard trace $\tr$ (respectively, the counting measure on
$\mathbb{Z}_+$), then it is not difficult to see that
$S(\mathcal{M})=S(\mathcal{M},\tau)=\mathcal{M}.$

\begin{definition}\label{mu}
Let a semifinite von Neumann  algebra $\mathcal M$ be equipped
with a faithful normal semi-finite trace $\tau$ and let $x\in
S(\mathcal{M},\tau)$. The generalized singular value function $\mu_\cM(x):t\rightarrow \mu_\cM(t;x)$ of
the operator $x$ is defined by setting
$$
\mu_\cM(t;x)
=
\inf\{\|xp\|_\infty:\ p=p^*\in\mathcal{M}\mbox{ is a projection,}\ \tau(\mathbf{1}-p)\leq t\}.
$$
When the algebra $(\cM,\tau)$ is clear from the context, we will use the notation $\mu(\cdot)$ instead of $\mu_\cM(\cdot)$. 
\end{definition}
An equivalent definition in terms of the
distribution function of the operator $x$ is the following. For every self-adjoint
operator $x\in S(\mathcal{M},\tau),$ setting
$$d_x(t)=\tau(E_{x}(t,\infty)),\quad t>0,$$
we have (see e.g. \cite{FK})
\begin{align}\label{dis}
\mu(t; x)=\inf\{s\geq0:\ d_{|x|}(s)\leq t\}.
\end{align}
It is clear that $\mu(x)=\mu(|x|)$ for $x\in S(\cM,\tau)$. Moreover,
\begin{equation}\label{mu_square}
\mu^2(x)=\mu(x^*x)=\mu(xx^*)
\end{equation}

It is well-known (see e.g. \cite[Corollary 2.3.16]{LSZ}) that 
\begin{equation}\label{mu_of_sum}
\mu(t+s, x+y)\leq \mu(t,x)+\mu(s,y),\quad x,y\in S(\cM,\tau), \quad t,s>0.
\end{equation}
For positive operators $x,y\in S(\cM,\tau)$ we have that 
\begin{equation}\label{mu^2_mu}
\mu(x^{1/2}yx^{1/2})=\mu^2(x^{1/2}y^{1/2})=\mu^2(y^{1/2}x^{1/2})=\mu(y^{1/2}xy^{1/2}).
\end{equation}
We also have that
\begin{equation}\label{prop_mu_projection}
\mu(|x|E_{|x|}(\alpha,\infty))=\mu(x)\chi_{[0,d_x(\alpha))}, \quad \alpha>0.
\end{equation}

If $\mathcal{M}=L^\infty(0,\infty)$, then  the
generalized singular value function $\mu(x)$ is precisely the
decreasing rearrangement $\mu(x)$ of the function $x$ (see e.g. \cite{KPS}) defined by
$$\mu(t;x)=\inf\{s\geq0:\ m(\{|x|\geq s\})\leq t\}.$$
If $\mathcal{M}=\cB(\cH)$ and $\tau$ is the
standard trace $\tr$, then 
$$\mu(n;x)=\mu(t; x),\quad t\in[n,n+1),\quad  n\geq0.$$
The sequence $\{\mu(n;x)\}_{n\geq0}$ is just the sequence of singular values of the operator $x$ \cite{GK}.

\begin{definition}\label{def:symmetric}
 A linear subspace $\cE(\cM)$ of $S(\cM,\tau)$ equipped with a complete (quasi-)norm $\|\cdot\|_\cE$, is called symmetric space (of $\tau$-measurable operators) if $x\in S(\cM,\tau)$, $y \in \cE(\cM)$ and $\mu(x)\le \mu(y)$ imply that $x\in \cE(\cM)$ and $\|x\|_\cE \le \|y\|_\cE$. In the case when $(\cM,\tau)$ is $(\cB(\cH),\tr)$ we denote $\cE(\cM)$ by $\cE(\cH)$.
\end{definition}
It is well-known that any symmetric space $\cE(\cM)$ is a normed $\cM$-bimodule, that is $axb\in \cE(\cM)$ for any $x\in \cE(\cM)$, $a,b\in \cM$ and 
$$\|axb\|_\cE\leq \|a\|_\infty\|b\|_\infty\|x\|_\cE.$$

The so-called Calkin correspondence provides a construction of symmetric operator spaces associated with the von Neumman algebra $\cM$ from concrete symmetric function spaces studied extensively in e.g. \cite{KPS}. Namely, let $(E(0,\infty),\|\cdot\|_{E(0 ,\infty)})$ be a symmetric function space on the semi-axis $(0,\infty)$. Then the pair
$$\cE(\cM)=\{x\in S(\cM,\tau):\mu(x)\in E(0,\infty)\},\quad \|x\|_{\cE(\cM)}:=\|\mu(x)\|_{E(0,\infty)}$$
is a symmetric space on $\cM$ \cite{Kalton_S} (see also \cite{LSZ}). For convenience, we denote $\|\cdot\|_{\cE(\cM)}$ by $\|\cdot\|_\cE$.

In particular, the noncommutative Schatten space $\cL_p(\cM)$, $p>0$, is the symmetric space corresponding to the classical $L_p$-space of functions $L_p(0,\infty)$, that is
$$\cL_p(\cM)=\{x\in S(\cM,\tau): \mu(x)\in L_p(0,\infty)\}.$$
This space can be also described as the space of all $\tau$-measurable operator $x$, such that $\tau(|x|^p)<\infty$ with $\|x\|_p=(\tau(|x|^p))^{1/p}.$

For the general $L_p$ spaces, the H\"older inequality holds as well, that is  if $x\in \cL_{p_1}(\cM)$, $y\in \cL_{p_1}(\cM)$, $\frac1{p_1}+\frac1{p_2}=\frac1q$, then $xy\in \cL_q(\cM)$ and 
\begin{equation}\label{Holder}
\|xy\|_q\leq \|x\|_{p_1}\|y\|_{p_2}.
\end{equation}
Moreover, if $x,y\in S(\cM,\tau)$ are such that $xy,yx\in \cL_1(\cM)$, then we also have that (\cite[Theorem 17]{Brown_Kosaki})
\begin{equation}\label{tr_AB}
\tau(xy)=\tau(yx).
\end{equation}

The symmetric space $S_0(\cM)$ of $\tau$-compact operators on $\cM$ is defined as
$$S_0(\cM)=\{x\in S(\cM,\tau): \lim_{s\to\infty}\mu(s,x)=0\}.$$

Another important example of symmetric spaces are weak  Schatten spaces, $p>0$, defined by 
$$\cL_{p,\infty}(\cM)=\{x\in S(\cM,\tau): \mu(x)\in L_{p,\infty}(0,\infty)\},$$
equipped with the quasi-norm
$$\|x\|_{p,\infty}=\sup_{t>0} t^{1/p}\mu(t,x),\quad x\in \cL_{p,\infty}(\cM).$$

Next, we recall the definition of tensor product of unbounded operators (see e.g. \cite{Kosaki}, \cite{KRII}). Let $x_1,x_2$ be closed densely defined operators on Hilbert spaces $\cH_1$ and $\cH_2$ with domains $\dom(x_1)$ and $\dom(x_2)$ respectively. The algebraic tensor product $x_1\otimes_{\rm alg}x_2$ of the operators $x_1$ and $x_2$ is a linear operator defined on the algebraic tensor product $\dom(x_1)\otimes_{\rm}\dom(x_2)$ by setting
$$(x_1\otimes_{\rm alg}x_2)\Big(\sum_{k=1}^n\xi_k\otimes \eta_k\Big)=\sum_{k=1}^n x_1\xi_k\otimes x_2\eta_k,\quad \xi_k\in\dom(x_1), \eta_k\in \dom(x_2).$$
The operator $(x_1\otimes_{\rm alg}x_2)$ is densely defined and closable, but not necessarily closed. The closure of the operator $(x_1\otimes_{\rm alg}x_2)$ is denoted by $x_1\otimes x_2$ and is called the tensor product of $x_1$ and $x_2$. If $x_1$ and $x_2$ are self-adjoint (respectively, positive), then the operator $x_1\otimes x_2$ is also self-adjoint (respectively, positive). 

Let $\cA_i$, $i=1,2$ be a semifinite von Neumann algebra on $\cB(\cH_i)$ and let $\tau_i$ be a faithful normal semifinite trace on $\cA_i$. We denote by  $(\cA_1\bar\otimes \cA_2, \tau_1\otimes \tau_2)$ the spatial tensor product of $(\cA_1,\tau_1)$ and $(\cA_2,\tau_2)$ (see e.g. \cite[Section 4.5]{Takesaki1}). If $x_i$ is affiliated with $\cA_i$, $i=1,2$, then $x_1\otimes x_2$ is affiliated with $\cA_1\bar\otimes \cA_2$. However, if $x_i\in S(\cA_i,\tau_i)$, $i=1,2$, then the inclusion  $x_1\otimes x_2\in S(\cA_1\bar\otimes \cA_2, \tau_1\otimes \tau_2)$ does not hold in general.

Let $E(0,\infty)$ be a symmetric function space. In what follows, the notation $x_1\otimes x_2\in \cE(\cA_1\bar\otimes \cA_2)$, means that $x_i\in S(\cA_i,\tau_i)$, $i=1,2$, are such that $x_1\otimes x_2\in S(\cA_1\bar\otimes \cA_2, \tau_1\otimes \tau_2)$ and $x_1\otimes x_2\in \cE(\cA_1\bar\otimes \cA_2).$

Note, that under the assumption $x_1\otimes x_2\in S(\cA_1\bar\otimes \cA_2, \tau_1\otimes \tau_2)$ the following equality holds (see e.g. \cite[Theorem 3.8]{Todorov})
\begin{equation}\label{mu_tensor}
\mu_{\cA_1\otimes\cA_2}(x_1\otimes x_2)=\mu\big(\mu_{\cA_1}(x_1)\otimes \mu_{\cA_2}(x_2)\big).
\end{equation}

For the special case of weak  Schatten spaces we have the following lemma (see e.g. \cite[Theorem 4.5]{Todorov}).
\begin{lemma}\label{tensor_L_1_weak_L_1}
If $x\in \cL_p(\cM)$ and $y\in \cL_{p,\infty}(\cN)$, then $x\otimes y\in \cL_{p,\infty}(\cM\otimes \cN)$ and 
$$\|x\otimes y\|_{p,\infty}\leq \|x\|_p\|y\|_{p,\infty}.$$
\end{lemma}
\begin{proof}
Let $z(t)=\frac1{t^{1/p}}$. By the definition of weak $L_p$ we have that 
$$\mu(y)\leq \|y\|_{p,\infty}z.$$ Therefore,
$$\|x\otimes y\|_{p,\infty}=\|x\otimes \mu(y)\|_{p,\infty}\leq \|y\|_{p,\infty}\|x\otimes z\|_{p,\infty}.$$
We have $\mu(x\otimes z)=\mu(\mu(x)\otimes \mu(z))=\|x\|_pz$, where the second equality follows from (the proof of) \cite[Theorem 2.f.2]{LTII}. Therefore, since $\|z\|_{p,\infty}=1$, the assertion follows.
\end{proof}

\subsection{Majorisation}
Let $x,y\in S(\mathcal{M},\tau).$ We say that $y$ is submajorized
by $x$ in the sense of Hardy-Littlewood-Polya if 
$$\int_0^t\mu(s,y)ds\leq\int_0^t\mu(s,x)ds,\quad\forall\, t>0.$$
In this case, we write $y\prec\prec x.$

It is well-known (see e.g. \cite[Theorem 3.3.3]{LSZ}) that 
\begin{equation}\label{submaj_sum}
\sum_{k=1}^n x_k\prec\prec \sum_{k=1}^n \mu(x_k),\quad x_k\in S(\cM,\tau), \, k=1,\dots n,\, n\in\bbN.
\end{equation}

The Hardy-Littlewood-Polya submajorization plays crucial role in our approach to abstract Cwikel estimates as it gives a necessary and sufficient condition for a symmetric function space to be interpolation space for $(L_p(0,\infty), L_\infty(0,\infty))$, $p\geq 1$.
Namely, it is well-known that a symmetric function space $E(0,\infty)$ is an (exact) interpolation space for $(L_1(0,\infty), L_\infty(0,\infty))$ if and only if condition $y\in E(0,\infty)$, $x\in (L_1+L_\infty)(0,\infty)$ and $x\prec\prec y$ implies that $x\in E(0,\infty)$ and $\|x\|_E\leq \|y\|_E$. The Lorentz-Shimogaki theorem \cite{LS_interpolation} (see also \cite{Sparr}) gives a similar description for interpolation spaces  for $(L_2(0,\infty), L_\infty(0,\infty))$. 

\begin{theorem}\label{interp}
Let $E(0,\infty)$ be a symmetric function space. The space $E(0,\infty)$ is an  interpolation space for the pair $(L_2(0,\infty), L_\infty(0,\infty))$ if and only if  condition $\mu^2(f)\prec\prec \mu^2(g), f\in S(0,\infty), g\in E(0,\infty)$ implies that $f\in E(0,\infty)$ and $\|f\|_E\leq C_E\|g\|_E$ for some constant $C$ independent of $f$ and $g$.
\end{theorem}

For convenience we recall the following elementary lemma (see e.g. \cite[(2.8)]{LS_interpolation}). 
\begin{lemma}\label{major_with_p}
Let $f,g\in (L_p+L_\infty)(0,\infty), p\geq 1$ and $f\prec\prec g$. Then $f^p\prec\prec g^p$. 
\end{lemma}
%\begin{proof}
%Fix $t>0$ and set 
%$$h_1=\mu(f)\chi_{(0,t)},\quad h_2=\mu(g)\chi_{(0,t)}.$$
%Since $f,g\in (L_p+L_\infty)(0,\infty)$, it follows that $h_1,h_2\in L_p(0,\infty)$ 
%The inequality $f\prec\prec g$ implies that $h_1\prec\prec h_2$. Therefore, since $L_p(0,\infty)$ is a fully symmetric space, we have that 
%$$\int_0^t\mu^p(s,f)ds=\|h_1\|_p^p\leq \|h_2\|_p^p=\int_0^t\mu^p(s,g)ds,$$
%as required.
%\end{proof}

The Lorentz-Shimogaki Theorem \ref{interp} allows us to prove (in particular) Cwikel estimates in ideals $\cL_p(\cH)$ and $\cL_{p,\infty}(\cH)$ for $p>2$ only. To obtain Cwikel estimate for $p<2$ we recall the notion of majorization of operators (see e.g. \cite{Marshall_Olkin}).

\begin{definition}\label{def_majorisation}
We say that $y$ is majorised by $x$  (notation $y\prec x$), if $y\prec\prec x$ and 
$$\int_0^\infty\mu(s,y)ds=\int_0^\infty\mu(s,x)ds,$$
assuming that both integrals are finite.
\end{definition}

The following proposition shows that in contrast to the case, when $p>2$, in the case when $0<p<2$, the spaces $\cL_p(\cM)$ and $\cL_{p,\infty}(\cM)$ respect the majorization in the reverse order. This fact appears to have escaped the attention of experts in this area. We refer to \cite{ASZ_pacific} where the germ of the idea used below can be discerned.

\begin{proposition}\label{clas_Lp_weakLp}
Let $0<p<2$ and let $x,y\in \cL_2(\cM)$ be such that $\mu^2(x)\prec \mu^2(y)$.
\begin{enumerate}
\item If $x\in\cL_{p}(\cM)$, then $y\in\cL_p(\cM)$ and  $\|y\|_p\leq \|x\|_p$.
\item If $x\in\cL_{p,\infty}(\cM)$, then $y\in\cL_{p,\infty}(\cM)$ and  $\|y\|_{p,\infty}\leq c_p \|x\|_{p,\infty}.$ 
\end{enumerate}
\end{proposition}
\begin{proof}
Due to Calkin correspondence it is sufficient to prove the assertion for the special case, when $\cM=L_\infty(0,\infty)$.

(i). Let $f,g\in L_2(0,\infty)$, $\mu^2(f)\prec \mu^2(g)$ and $f\in L_p(0,\infty)$. 
Without loss of generality, we can assume that $f=\mu(f),\,  g=\mu(g)$. Let $n\in\bbN$. Since 
$\int_0^{n}g^2(s)ds\geq \int_0^nf^2(s)ds$, there exists $a_n\in (0,n]$ such that 
$$\int_0^{a_n}g^2(s)ds=\int_0^n f^2(s)ds.$$

We set 
\begin{align*}
f_n=f\chi_{(0,n)},\quad g_n=g\chi_{(0,a_n)}.
\end{align*}

We claim that $f_n^2\prec g_n^2$. 

Let $0<t\leq a_n$. Since $a_n\leq n$ we have 
$$\int_0^tf^2_n(s)ds=\int_0^t f^2(s)ds\leq \int_0^t g^2(s)ds=\int_0^t g^2_n(s)ds.$$
If $t>a_n$, then 
$$\int_0^t f_n^2(s)ds\leq \int_0^n f^2(s)ds=\int_0^{a_n}g^2(s)ds=\int_0^t g^2_n(s)ds.$$
In addition, the choice of $a_n$ ensures that 
$$\int_0^\infty f_n^2(s)ds=\int_0^\infty g_n^2(s)ds.$$
Thus, $f_n^2\prec g_n^2$ as required. 

By \cite[Lemma 25]{ASZ_pacific} (scaling the Lebesgue measure on $(0,n)$) we have that 
$$\|g_n\|_p=\|g_n^2\|_{p/2}^{\frac12}\leq \|f_n^2\|_{p/2}^{\frac12}=\|f_n\|_p.$$ 
Passing to the limit as $n\to\infty$ we infer the inequality 
$$\|g\|_p\leq \|f\|_p,$$
as required.

(ii). Let $f,g\in L_2(0,\infty)$, $\mu^2(f)\prec \mu^2(g)$ and $f\in L_{p,\infty}(0,\infty)$. Without loss of generality, we may assume that $f=\mu(f), g=\mu(g)$.

Fix $t>0$. Since 
$$\int_0^\infty f^2(s)ds=\int_0^\infty g^2(s)ds$$
 and 
 $$\int_0^t f^2(s)ds\leq \int_0^t g^2(s)ds$$ we have
\begin{align}\label{eq1}
\int_t^\infty g^2(s)ds\leq  \int_t^\infty f^2(s)ds
\end{align}
Since $f=\mu(f)\in L_{p,\infty}(0,\infty)$ we have that $f^2\in L_{p/2,\infty}(0,\infty)$ and 
$$f^2(s)\leq \|f^2\|_{\frac{p}2,\infty} s^{-2/p},\quad s>0.$$ Therefore, since $p<2$, we have 
\begin{equation}\label{eq2}
\int_t^\infty f^2(s)ds\leq \|f^2\|_{p/2,\infty}\int_t^\infty s^{-2/p}ds=\const \|f\|^2_{p,\infty} t^{1-2/p}.
\end{equation}

Since $g^2$ is decreasing we have 
\begin{align*}
tg^2(2t)\leq \int_t^\infty g^2(s)ds\stackrel{\eqref{eq1}}{\leq}\int_t^\infty f^2(s)ds
&\stackrel{\eqref{eq2}}{\leq}\const \|f\|^2_{p,\infty} t^{1-2/p}.
\end{align*}
Hence, dividing both sides by $t$ and taking the square roots we obtain that 
$$g(2t)\leq \const \|f\|_{p,\infty} t^{1/p},$$
or equivalently 
$$g(t)\leq 2^{-1/p} \const \|f\|_{p,\infty} t^{1/p}.$$

Thus, $g\in L_{p,\infty}(0,\infty)$ and 
$$\|g\|_{p,\infty}\leq c_p \|f\|_{p,\infty}.$$
\end{proof}

Next, we recall the definition of direct sum of operators (see \cite{LSZ}). 
\begin{definition}Let $x_k\in S(\cM,\tau)$, $k\geq 0$. If $\{p_k\}\subset \cM$ is a sequence of pairwise orthogonal projections and if $y_k\in p_kS(\cM,\tau)p_k$ are such that $\mu(y_k)=\mu(x_k)$, then we write 
$$\bigoplus_{k} x_k:=\sum_k y_k.$$
\end{definition}
The particular choice of operators $y_k, k\geq 0$ is inessential, we are only interested in the fact that 
$$\mu\Big(\bigoplus_{k} x_k\Big)=\mu\Big(\bigoplus_{k} \mu(x_k)\Big).$$
Hence, if $x_k\prec y_k$, $k\in\bbN$, then 
\begin{equation}\label{strong_major_bigoplus}
\bigoplus x_k\prec \bigoplus y_k.
\end{equation}

The following lemma allows us to compare the direct sum of operators with series of these operators in term of the majorisation defined in Definition \ref{def_majorisation}. 
\begin{lemma}\label{lem_strong_major}
\begin{enumerate}
\item Let $T_k\in \cL_2(\cM)$ be such that $\sum_k \|T_k\|_2^2<\infty.$ 
If $\{T_k\}$ are disjoint from the right or from the left, then 
$$\mu^2(\bigoplus_k T_k)\prec \mu^2(\sum_k T_k).$$
\item Let $x_k,y_k\in S(\cM,\tau)$ be such that $\sum_{k,l}\|x_ky_l\|_2^2<\infty$. If $\{y_k\}$ are disjoint from the right and $\{x_k\}$ is disjoint from the left, then 
$$\mu^2(\bigoplus_{k,l} x_ky_l)\prec \mu^2(\sum_{k,l}x_ky_l).$$
\end{enumerate}
\end{lemma}
\begin{proof}
(i). It is sufficient to prove the assertion under the assumption that $T_k$ are disjoint from the right, since the assertion for operators disjoint from the left can be proved similarly by taking adjoins.

We firstly note, that since $T_k$ are disjoint from the right, it follows that 
$$\Big\|\sum_{k=n}^m T_k\Big\|_2^2=\tau\Big(\sum_{k,l=n}^m T_k^*T_l\Big)=\sum_{k=n}^m\tau\Big( T_k^*T_k\Big)=\sum_{k=n}^m\|T_k\|_2^2,$$ 
whenever  $m>n,\,  n,m\in\bbN$.
Since the series $\sum_k \|T_k\|_2^2$ converges, we obtain that the series $\sum_k T_k$ converges in $\cL_2(\cM)$, and moreover,
$$\Big|\sum_k T_k\Big|^2=\sum_{k,l} T_k^*T_l=\sum_{k}T_k^*T_k.$$
By \cite[Lemma 3.3.7]{LSZ} (see also \cite{CDS}) we have that 
$$\mu^2\Big(\bigoplus_k T_k\Big)=\mu\Big(\bigoplus_k |T_k|^2\Big)\prec\prec \mu\Big(\sum_{k}T_k^*T_k\Big)=\mu^2\Big(\sum_k T_k\Big),$$
which proves submajorisation. 

Since the series $\sum_k T_k$ converges in $\cL_2(\cM)$, we have 
$$\int_0^\infty \mu^2\Big(s, \bigoplus_k T_k\Big)ds=\Big\|\bigoplus_k T_k\Big\|_2^2=\Big\|\sum_k T_k\Big\|_2^2=\int_0^\infty \mu^2\Big(s, \sum_k T_k\Big)ds,$$
which concludes the proof.

(ii). By the assumption we have that the operators $\{x_ky_l\}_{k,l}$ are disjoint from the left and from the right. Hence, repeating the same argument as in part (i), we have that the series $\sum_{k,l} x_ky_l$ converges in $\cL_2(\cM)$. Since for any fixed $k$ the operators $x_ky_l$ are disjoint from the right, part (i) implies that 
$$\mu^2\Big(\bigoplus_{l}x_ky_l\Big)\prec \mu^2\Big(\sum_{l} x_ky_l\Big),\quad k\in\bbN.$$
Therefore, by Lemma \ref{strong_major_bigoplus} we have that 
$$\mu^2\Big(\bigoplus_{k,l}x_ky_l\Big)=\mu^2\Big(\bigoplus_{k}\bigoplus_lx_ky_l\Big)\prec \mu^2\Big(\bigoplus_k(\sum_l x_ky_l)\Big).$$
Using part (i) again for operator $\sum_l x_ky_l$, $k\in\bbN$, which are disjoint from the left, we obtain that 
$$\mu^2\Big(\bigoplus_{k,l}x_ky_l\Big)\prec \mu^2\Big(\sum_k(\sum_l x_ky_l)\Big)=\mu^2\Big(\sum_{k,l} x_ky_l\Big).$$
\end{proof}

\section{Abstract Cwikel estimates}
In this section we prove abstract Cwikel estimates for interpolation spaces in the Banach couple $(L_2,L_\infty)$. In particular, we obtain the estimate for the ideal  $\cL_{p,\infty}(\cH)$ for any $p>2$. 

Let $\cA_1$ and $\cA_2$ be semifinite von Neumann algebras represented on the same (separable) Hilbert space $\cH$ and let $\tau_1$ and $\tau_2$ be faithful normal semifinite traces on $\cA_1$ and $\cA_2,$ respectively. We denote by $\pi_1$ (respectively, $\pi_2$) the representation of $\cA_1$ (respectively, $\cA_2$) on $\cH$. By the same argument as in \cite[Lemma 2.4]{dePS_JFA_2004} (and the paragraph after the lemma) we have that the representation $\pi_i$, $i=1,2$ can be extended uniquely to a representation of $S(\cA_i,\tau_i)$ on $\cH$. In what follows we  use the same notation $\pi_i$ for these extensions.

Throughout this section we assume that Cwikel type estimates hold for the Hilbert-Schmidt ideal $\cL_2(\cH)$, that is we have the following
\begin{hypothesis}\label{abs_Cwikel_hyp}
Assume that for every $x\in \cL_2(\cA_1,\tau_1)$ and $y\in \cL_2(\cA_2,\tau_2)$ we have $\pi_1(x)\pi_2(y)\in \cL_2(\cH)$ and
$$\|\pi_1(x)\pi_2(y)\|_{\cL_2(\cH)}\leq \const\|x\|_{\cL_2(\cA_1)}\|y\|_{\cL_2(\cA_2)}.$$
Without loss of generality we may assume that $\const=1$.
\end{hypothesis}

\begin{lemma}\label{abs_Cwikel_proj_ineq}
Suppose that $0\leq x\in S(\cA_1,\tau_1)$ and  $0\leq y\in S(\cA_2,\tau_2)$. Then for every $n\in\bbZ$ we have 
$$E_{x\otimes y}[2^n,\infty)\geq \sum_{k+l\geq n}E_{x}[2^k,2^{k+1})\otimes E_{y}[2^l,2^{l+1}),\quad k,l\in\bbZ.$$ 
\end{lemma}
\begin{proof}Fix $k\in\bbZ$ and set 
$$x_{k}:=xE_{x}[2^k,2^{k+1}),\quad y_{k}:=yE_{y}[2^k,2^{k+1}).$$
Taking into account that $(x_k\otimes y_l)(x_{k'}\otimes y_{l'})=0$ unless $k=k'$, $l=l'$, we have 
$$E_{x\otimes y}[2^n,\infty)=\sum_{k,l\in\mathbb{Z}}E_{x_{k}\otimes y_{l}}[2^n,\infty)\geq\sum_{k+l\geq n}E_{x_{k}\otimes y_{l}}[2^n,\infty).$$
By the definition of $x_k$ and $y_l$, for $k+l\geq n$ we have
$$x_{k}\otimes y_{l}\geq 2^kE_{x}[2^k,2^{k+1})\otimes 2^lE_{y}[2^l,2^{l+1})\geq 2^nE_{x}[2^k,2^{k+1})\otimes E_{y}[2^l,2^{l+1}).$$
Therefore,
$$E_{x\otimes y}[2^n,\infty)\geq \sum_{k+l\geq n}E_{x}[2^k,2^{k+1})\otimes E_{y}[2^l,2^{l+1})$$
as required.
\end{proof}

\begin{lemma}\label{abs_Cwikel_finitely_many}
Assume that $x\in\cA_1\cap S_0(\cA_1,\tau_1)$, $y\in \cA_2\cap S_0(\cA_2,\tau_2)$ are positive. If $x\otimes y\in (\cL_2+\cL_\infty)(\cA_1\otimes \cA_2)$, then 
$$\mu^2_{\cB(\cH)}(\pi_1(x)\pi_2(y))\prec\prec 130 \mu^2_{\cA_1\otimes \cA_2}(x\otimes y).$$
In particular, 
\begin{equation}\label{abs_Cwikel_uniform_norm_bounded}
\|\pi_1(x)\pi_2(y)\|_\infty\leq 130 \|x\otimes y\|_{(\cL_2+\cL_\infty)(\cA_1\otimes \cA_2)}.
\end{equation}
\end{lemma}
\begin{proof}
Fix $t>0$ and choose $n\in\bbZ$ such that
\begin{equation}\label{abs_Cwikel_choose_n}
\mu_{\cA_1\otimes \cA_2}(t,x_1\otimes x_2)<2^n\leq 2\mu_{\cA_1\otimes \cA_2}(t,x_1\otimes x_2).
\end{equation} 
 By the definition of singular value function (see \eqref{dis}) the first inequality implies that 
\begin{equation}\label{abs_Cwikel_trace_E_n}
\tau_{\cA_1\otimes\cA_2}(E_{x\otimes y}[2^n,\infty))\leq t.
\end{equation}

In order to estimate $\int_0^t\mu_{\cB(\cH)}^2(s,\pi_1(x)\pi_2(y))ds$ we decompose $x$ and $y$ in dyadic parts in the following way
\begin{equation}
x_{k}:=xE_{x}[2^k,2^{k+1}),\quad y_{k}:=yE_{y}[2^k,2^{k+1}),\quad k\in\mathbb{Z}.
\end{equation}
By the assumptions there exists $N\in \bbN$ such that $x_k,y_k=0$ for $k\geq N$.

We have (with respect to the weak-operator topology)
\begin{align}\label{abs_Cwikel_into_series}
\pi_1(x)\pi_2(y)&=\sum_{k,l\in\mathbb{Z}}\pi_1(x_{k})\pi_2(y_{l})=\sum_{k+l<n}\pi_1(x_{k})\pi_2(y_{l})+\sum_{k+l\geq n}\pi_1(x_{k})\pi_2(y_{l})\nonumber\\
&=:A_n+B_n,\quad n\in\bbZ.
\end{align}
By the assumptions, the operator $B_n$ has finitely many summands. We claim that the series $A_n$ converges in the uniform norm. We also estimate the uniform norm of  $A_n$ and Hilbert-Schmidt norm of $B_n$. 

Note, that since $x,y$ are $\tau_i$ compact, respectively, it follows that $x_k\in\cL_2(\cA_1,\tau_1)$ and $y_l\in\cL_2(\cA_2,\tau_2)$ for all $k,l\in\bbZ$. Therefore, Hypothesis \ref{abs_Cwikel_hyp} implies that $\pi_1(x_{k})\pi_2(y_{l})\in \cL_2(\cH)$ for every $k,l\in \bbZ$.

For the operator $A_n$ we write
$$A_n=\sum_{k+l<n}\pi_1(x_{k})\pi_2(y_{l})=\sum_{m<n}\sum_{k+l=m}\pi_1(x_{k})\pi_2(y_{l}).$$
We have 
$$\Big|\sum_{k+l=m}\pi_1(x_{k})\pi_2(y_{l})\Big|^2=\sum_{\substack{k_1+l_1=m\\ k_2+l_2=m}}\pi_2(y_{l_2})\pi_1(x_{k_2}x_{k_1})\pi_2(y_{l_1}).$$
Since $x_{k_2}x_{k_1}=0$ when $k_1\neq k_2$, we have 
$$\Big|\sum_{k+l=m}\pi_1(x_{k})\pi_2(y_{l})\Big|^2=\sum_{\substack{k+l_1=m\\ k+l_2=m}}\pi_2(y_{l_2})\pi_1(x_{k}^2)\pi_2(y_{l_1})=\sum_{k+l=m}\pi_2(y_{l})\pi_1(x_{k}^2)\pi_2(y_{l}).$$
By definition of $\{x_k\},\{y_l\}$ we have 
\begin{align*}
\sup_{k+l=m}\|\pi_2(y_{l})\pi_1(x_{k}^2)\pi_2(y_{l})\|_{\infty}&\leq \sup_{k+l=m}\|x_{k}\|_{\infty}^2\|y_{l}\|_{\infty}^2\leq \sup_{k+l=m}2^{2(k+1)}\cdot 2^{2(l+1)}\\
&=2^{2m+4}.
\end{align*}
Hence, the fact that the operators $\pi_2(y_{l})\pi_1(x_{k}^2)\pi_2(y_{l}), k,l\in\bbZ$, are pairwise disjoint implies the operator $\Big|\sum_{k+l=m}\pi_1(x_{k})\pi_2(y_{l})\Big|^2$ is bounded and 
\begin{align*}
\Big\|\sum_{k+l=m}\pi_1(x_{k})\pi_2(y_{l})\Big\|_{\infty}^2&=\Big\|\Big|\sum_{k+l=m}\pi_1(x_{k})\pi_2(y_{l})\Big|^2\Big\|_\infty\\
&=\sup_{k+l=m}\|\pi_2(y_{l})\pi_1(x_{k}^2)\pi_2(y_{l})\|_{\infty}\leq 2^{2m+4}.
\end{align*}

Therefore, 
\begin{align*}
\sum_{m<n}\Big\|\sum_{k+l=m}\pi_1(x_{k})\pi_2(y_{l})\Big\|_{\infty}\leq\sum_{m<n}2^{m+2}=2^{n+2}.
\end{align*}
Thus, the series $A_n=\sum_{m<n}\sum_{k+l=m}\pi_1(x_{k})\pi_2(y_{l})$ converges in the uniform norm and 
\begin{align}\label{abs_Cwikel_norm_An}
\|A_n\|_\infty=\Big\|\sum_{m<n}\sum_{k+l=m}\pi_1(x_{k})\pi_2(y_{l})\Big\|_{\infty}
\leq 2^{n+2}.
\end{align}

Next, since the sum in $B_n$ has finitely many non-zero summands, we have
$$\|B_n\|_2^2=\tr(B_n^*B_n)=\sum_{k_1+l_1\geq n}\sum_{k_2+l_2\geq n}\tr(\pi_2(y_{l_2})\pi_1(x_{k_2}x_{k_1})\pi_2(y_{l_1})).$$
The fact that 
$$\pi_2(y_{l_2})\pi_1(x_{k_2}),\, \pi_1(x_{k_1})\pi_2(y_{l_1}) \in \cL_2(\cH), \quad k,l\in\bbZ$$ together with equality \eqref{tr_AB} implies that 
$$\tr\big([\pi_2(y_{l_2})\pi_1(x_{k_2})][\pi_1(x_{k_1})\pi_2(y_{l_1})]\big)=\tr\big([\pi_1(x_{k_1})\pi_2(y_{l_1})][\pi_2(y_{l_2})\pi_1(x_{k_2})]\big)=0,$$
whenever $(k_1,l_1)\neq(k_2,l_2).$
Since $x_k, y_l$ are positive, Hypothesis \ref{abs_Cwikel_hyp} implies that for $(k_1,l_1)=(k_2,l_2),$ we have 
$$\tr(\pi_2(y_{l_1})\pi_1(x_{k_1}^2)\pi_2(y_{l_1}))=\|\pi_1(x_{k_1})\pi_2(y_{l_1})\|_2^2\leq\|x_{k_1}\|_2^2\|y_{l_1}\|_2^2.$$
Therefore, we have
$$\|B_n\|_2^2\leq\sum_{k+l\geq n}\|x_{k}\|_2^2\|y_{l}\|_2^2=\sum_{k+l\geq n}\|x_{k}\otimes y_{l}\|_2^2.$$

Setting $$e_n=\sum_{k+l\geq n}E_{x}[2^k,2^{k+1})\otimes E_{y}[2^l,2^{l+1})\in \cA_1\otimes\cA_2,$$ 
we have 
$$\sum_{k+l\geq n}\|x_{k}\otimes y_{l}\|_2^2=\|(x\otimes y)\cdot e_n\|_2^2.$$
Lemma \ref{abs_Cwikel_proj_ineq} combined with \eqref{abs_Cwikel_trace_E_n} implies that $(\tau_1\otimes \tau_2)(e_n)\leq t$, and therefore, 
$$(x\otimes y)\ e_n\in \cL_2(\cA_1\otimes\cA_2,\tau_1\otimes \tau_2)$$
and 
$$\|(x\otimes y)\ e_n\|_2\leq \|(x\otimes y)E_{x\otimes y}[2^n,\infty)\|_2<\infty.$$
Thus,
\begin{equation}\label{abs_Cwikel_norm_Bn}
\|B_n\|_2^2\leq\|(x\otimes y)\cdot E_{x\otimes y}[2^n,\infty)\|_2^2.
\end{equation}

Having obtained the estimates on the norms of $A_n$ and $B_n$ we can now estimate the required singular value function.
By \eqref{abs_Cwikel_into_series} and \eqref{submaj_sum} we have 
$$\mu_{\cB(\cH)}(\pi_1(x)\pi_2(y))\prec\prec\mu_{\cB(\cH)}(A_n)+\mu_{\cB(\cH)}(B_n).$$ Hence, by Lemma \ref{major_with_p} we infer that
\begin{align*}
\int_0^t&\mu^2_{\cB(\cH)}(s,\pi_1(x)\pi_2(y))ds\leq\int_0^t\Big(\mu_{\cB(\cH)}(s,A_n)+\mu_{\cB(\cH)}(s,B_n)\Big)^2ds\\
&\leq 2\int_0^t\Big(\mu_{\cB(\cH)}^2(s,A_n)+\mu_{\cB(\cH)}^2(s,B_n)\Big)ds\leq 2t\|A_n\|_{\infty}^2+ 2\int_0^\infty\mu_{\cB(\cH)}^2(s,B_n)ds\\
&= 2t\|A_n\|_{\infty}^2+2\|B_n\|_2^2.
\end{align*}
Therefore, by \eqref{abs_Cwikel_norm_An} and \eqref{abs_Cwikel_norm_Bn} we have 
$$\int_0^t\mu_{\cB(\cH)}^2(s,\pi_1(x)\pi_2(y))ds\leq 2^{2n+5}t+2\|(x\otimes y)E_{x\otimes y}[2^n,\infty)\|_2^2.$$
By \eqref{prop_mu_projection} and \eqref{abs_Cwikel_trace_E_n} we have that 
\begin{align*}
\|(x\otimes y)E_{x\otimes y}[2^n,\infty)\|_2^2&\stackrel{\eqref{prop_mu_projection}}{=}\int_0^{\tau_{\cA_1\otimes \cA_2}(E_{x\otimes y}[2^n,\infty))}\mu_{\cA_1\otimes \cA_2}^2(s,x\otimes y)ds\\
&\stackrel{\eqref{abs_Cwikel_trace_E_n}}\leq \int_0^{t}\mu_{\cA_1\otimes \cA_2}^2(s,x\otimes y)ds.
\end{align*}
Hence, combining the above inequalities with the second inequality in 
\eqref{abs_Cwikel_choose_n} we obtain that 
\begin{align*}
\int_0^t\mu^2_{\cB(\cH)}(s,\pi_1(x)\pi_2(y))ds&\leq 2^7t\mu_{\cA_1\otimes \cA_2}^2(t,x\otimes y)+2\int_0^t\mu_{\cA_1\otimes \cA_2}^2(s,x\otimes y)ds\\
&\leq 130\int_0^t\mu_{\cA_1\otimes \cA_2}^2(s,x\otimes y)ds,
\end{align*}
that is 
$$\mu^2_{\cB(\cH)}(\pi_1(x)\pi_2(y))\prec\prec 130 \mu^2_{\cA_1\otimes \cA_2}(x\otimes y),$$
as required.

To prove estimate \eqref{abs_Cwikel_uniform_norm_bounded}, we note that since $\mu_{\cB(\cH)}^2(\pi_1(x)\pi_2(y))$ is a step function, we obtain 
\begin{align*}
\|\pi_1(x)\pi_2(y)\|_\infty^2& = \int_0^{1}\mu^2_{\cB(\cH)}(s,\pi_1(x)\pi_2(y)) ds\leq 130\int_0^{1}\mu^2_{\cA_1\otimes \cA_2}(s,x\otimes y) ds\\
&\leq 130 \|x\otimes y\|^2_{(\cL_2+\cL_\infty)(\cA_1\otimes \cA_2)}.
\end{align*}

\end{proof}

In the following theorem we extend Lemma \ref{abs_Cwikel_finitely_many} to general $x,y$, such that $x\otimes y \in (\cL_2+\cL_{\infty})(\cA_1\otimes \cA_2,\tau_1\otimes \tau_2)$.
\begin{theorem}\label{estim_for_2} Suppose that  $x\otimes y \in (\cL_2+\cL_{\infty})(\cA_1\otimes \cA_2,\tau_1\otimes \tau_2)$. Then $\pi_1(x)\pi_2(y)\in \cB(\cH)$ and 
$$\mu_{\cB(\cH)}^2(\pi_1(x)\pi_2(y))\prec\prec 532\mu_{\cA_1\otimes \cA_2}^2(x\otimes y).$$
\end{theorem}
\begin{proof}
By equality \eqref{mu_square} we have that 
\begin{align*}
\mu^2_{\cB(\cH)}&(\pi_1(x)\pi_2(y))=\mu_{\cB(\cH)}(\pi_2(y^*)\pi_1(x^*x)\pi_2(y))\\
&=\mu_{\cB(\cH)}(\pi_2(y^*)\pi_1(|x|)\pi_1(|x|)\pi_2(y))=\mu_{\cB(\cH)}(\pi_1(|x|)\pi_2(y)\pi_2(y^*)\pi_1(|x|))\\
&=\mu^2_{\cB(\cH)}(\pi_1(|x|)\pi_2(|y^*|)).
\end{align*}

On the other hand, 
\begin{align*}
\mu^2_{\cA_1\otimes \cA_2}(x\otimes y)&=\mu_{\cA_1\otimes \cA_2}(|x|^2\otimes |y|^2)\stackrel{\eqref{mu_tensor}}{=}\mu_{\cA_1\otimes \cA_2}(\mu(|x|^2)\otimes \mu(|y|^2))\\
&=\mu_{\cA_1\otimes \cA_2}(\mu^2(x)\otimes \mu^2(y^*))\\
&=\mu_{\cA_1\otimes \cA_2}(|x|^2\otimes |y^*|^2)=\mu^2_{\cA_1\otimes \cA_2}(|x|\otimes |y^*|).
\end{align*}
Thus, without loss of generality we can assume that $x,y\geq0.$

At the first step we suppose that $x\in S_0(\cA_1,\tau_1)$, $y\in \cA_2\cap S_0(\cA_2,\tau_2)$ are such that $x\otimes y \in (\cL_2+\cL_{\infty})(\cA_1\otimes \cA_2,\tau_1\otimes \tau_2)$. 
Let us denote $p_n=E_x[0,n]$ and $C_n=\pi_1(p_n x)\pi_2(y) .$ By Lemma \ref{abs_Cwikel_finitely_many} we have that 
\begin{equation}\label{abs_Cwikel_first_reduction}
\mu^2_{\cB(\cH)}(C_n)\prec\prec 130 \mu^2_{\cA_1\otimes \cA_2}(p_nx\otimes y)\leq \mu^2_{\cA_1\otimes \cA_2}(x\otimes y).
\end{equation}

Furthermore, by \eqref{abs_Cwikel_uniform_norm_bounded} we have 
\begin{align*}
\|C_n\|_\infty\leq
 \const \|(p_nx)\otimes y\|_{(\cL_2+\cL_\infty)(\cA_1\otimes \cA_2)}\leq \const \|x\otimes y\|_{(\cL_2+\cL_\infty)(\cA_1\otimes \cA_2)}:=c.
\end{align*}
For arbitrary $\xi,\eta\in H$ and $n\geq m\in\bbN$ we have 
$$\langle (C_n-C_m)\xi,\eta\rangle=\langle\pi_1(1-p_m)C_n\xi,\eta\rangle=\langle C_n\xi, \pi_1(1-p_m)\eta\rangle.$$
Since $1-p_m\downarrow 0$ it follows that (see e.g. \cite[III.2.2.1]{Blackadar}) $\pi_1(1-p_m)\downarrow 0$. Therefore 
$$|\langle (C_n-C_m)\xi,\eta\rangle|\leq \|C_n\xi\|\|\pi_1(1-p_m)\eta\|\leq c\|\xi\|\|\pi_1(1-p_m)\eta\|\to 0\quad n,m\to\infty.$$

Thus, the bilinear form $(\xi,\eta)\mapsto \lim_{n\to\infty}\langle C_n\xi,\eta\rangle$ is bounded on $\cH\times \cH$. Therefore, there exists $C$ such that $C_n\to C$ in the weak operator topology. It is clear that $C_n=\pi_1(p_n)C$. Therefore $C=\pi_1(x)\pi_2(y)$. Thus $\pi_1(p_n x)\pi_2(y)\to \pi_1(x)\pi_2(y)$ in the weak operator topology, which together with \eqref{abs_Cwikel_first_reduction} implies that 
$$\mu^2_{\cB(\cH)}(\pi_1(x)\pi_2(y))\prec\prec 130 \mu^2_{\cA_1\otimes \cA_2}(x\otimes y).$$

%
%Next, we remove the assumption that $y\in \cA_2$ (but we still  assume that $x,y$ are $\tau_i$-compact operators, respectively, with $x\otimes y \in (\cL_2+\cL_{\infty})(\cA_1\otimes \cA_2,\tau_1\otimes \tau_2)$). We set  $q_n=E_y[0,n]$ and $D_n=\pi_1(x)\pi_2(yq_n).$ By the previous estimate we get
%$$\mu^2_{\cB(\cH)}(D_n)\prec\prec 130 \mu^2_{\cA_1\otimes \cA_2}(x\otimes yq_n)\prec\prec130 \mu^2_{\cA_1\otimes \cA_2}(x\otimes y).$$ Since $x\otimes y\in (\cL_2+ \cL_\infty)(\cA_1\otimes \cA_2)$, for any $1>t>0$ we have 
%$$\mu^2_{\cB(\cH)}(t,D_n)=\frac1t\int_0^t\mu^2_{\cB(\cH)}(s,\pi_1(x)\pi_2(yq_n))ds\leq \frac{130}t  \int_0^t\mu^2(s,x\otimes y)ds<\infty.$$
%Since $\mu_{\cB(\cH)}^2(D_n)$ is a step function, we have that $\mu_{\cB(\cH)}^2(0,D_n)=\mu_{\cB(\cH)}^2(\frac12,D_n)<\infty$, which implies that $\{D_n\}$ is a bounded  sequence in $\cB(\cH)$. Repeating the argument above, one can obtain that $D_n\to \pi_1(x)\pi_2(y)$ in the weak operator topology. Hence, 
%$$\mu_{\cB(\cH)}^2(\pi_1(x)\pi_2(y))\prec\prec 130 \mu_{\cA_1\otimes \cA_2}^2(x\otimes y).$$

Repeating now the same argument, we can remove the assumption that $y\in \cA_2$. Hence, we assume now 
that $x,y\geq0$ are arbitrary with $x\otimes y \in (\cL_2+\cL_{\infty})(\cA_1\otimes \cA_2,\tau_1\otimes \tau_2)$ and let 
\begin{align*}
a=\mu(\infty,x), \quad b=\mu(\infty,y),\\
u=(x-a)_+,\quad v=(y-b)_+.
\end{align*}

We have 
\begin{align}\label{abs_Cwikel_no_compactness}
\mu_{\cB(\cH)}(\pi_1(x)\pi_2(y))&\leq \mu_{\cB(\cH)}(\pi_1(u+a)\pi_2(v+b))\nonumber\\
&\prec\prec \mu_{\cB(\cH)}(\pi_1(u)\pi_2(v))+a\mu_{\cB(\cH)}(\pi_2(v))+b\mu_{\cB(\cH)}(\pi_1(u))+ab.
\end{align}

Let $b\neq 0$ and assume that $x$ is not a bounded operator. Let $q=E_y(\frac12 b,\infty)$. Since $x\otimes y$ is $\tau_1\otimes \tau_2$-measurable, we have that $x\otimes q$ is also $\tau_1\otimes \tau_2$-measurable. Since $b= \mu(\infty,y)\neq 0$ we have that $\tau_2(q)=\infty$. Hence, 
$$(\tau_1\otimes \tau_2)(E_{x\otimes q}(n,\infty))=(\tau_1\otimes \tau_2)(E_{x}(n,\infty)\otimes q)=\infty,$$
which contradicts to $\tau_1\otimes \tau_2$-measurability of 
$x\otimes q$, and therefore, $E_{x}(n,\infty)=0$ for some $n\in\bbN$.
Thus, if  $b\neq 0$, then $x$ is a bounded operator. In this case, since $\tau_2(1)=\infty$, it follows that
$$\mu^2_{\cB(\cH)}(\pi_1(u))\leq \|u\|_\infty^2=\mu^2_{\cA_1\otimes \cA_2}(u\otimes 1).$$
Similarly, if $a\neq 0$, then $y$ is bounded and $\mu^2_{\cB(\cH)}(\pi_2(v))\leq \mu^2_{\cA_1\otimes \cA_2}(1\otimes v).$

We can have three different cases: $a=0,b\neq0$;  $a\neq 0$, $b=0$ and both $a,b\neq0$.
We consider only the case, when $a,b\neq0$, as the other two cases can be proved similarly.
%In the case, when 
% by \eqref{abs_Cwikel_no_compactness} we have 
%\begin{align*}
%\mu^2_{\cB(\cH)}(\pi_1(x)\pi_2(y))&\prec\prec \Big(\mu_{\cB(\cH)}(\pi_1(u)\pi_2(v))+b\mu_{\cB(\cH)}(\pi_1(u))\Big)^2\\
%&\leq 2\big(\mu^2_{\cB(\cH)}(\pi_1(u)\pi_2(v))+b^2\mu^2_{\cB(\cH)}(\pi_1(u))\big)\\
%&\prec\prec 2\big(130\mu^2_{\cA_1\otimes \cA_2}(u\otimes v)+b^2\mu^2_{\cA_1\otimes \cA_2}(u\otimes 1)\big)\leq 262\mu^2(x\otimes y).
%\end{align*}
%The case $a\neq 0$, $b=0$ can be established using the same argument.
By \eqref{abs_Cwikel_no_compactness}  we have 
\begin{align*}
\mu^2_{\cB(\cH)}(\pi_1(x)\pi_2(y))&\prec\prec 4(\mu^2(\pi_1(u)\pi_2(v))+a^2\mu^2(\pi_2(v))+b^2\mu^2(\pi_1(u))+a^2b^2)\\
&\prec\prec 4(130\mu^2_{\cA_1\otimes \cA_2}(u\otimes v)+3\mu^2_{\cA_1\otimes \cA_2}(x\otimes y))\\
&\leq 532\mu^2_{\cA_1\otimes \cA_2}(x\otimes y).
\end{align*}
\end{proof}

Combining this result with Theorem \ref{interp} (with $f=\mu_{\cB(\cH)}(\pi_1(x)\pi_2(y))$ and $g=\mu_{\cA_1\otimes \cA_2}(x\otimes y)$)  we arrive at the following 
\begin{corollary}\label{abs_Cwikel_main_thm}
 Let $(E(0,\infty),\|\cdot\|_E)$ be an interpolation space for $(L_2,L_\infty)$ and let $x\in S(\cA_1,\tau_1), y\in S(\cA_2,\tau_2)$ be such that $x\otimes y\in \cE(\cA_1\otimes \cA_2)$. Then $\pi_1(x)\pi_2(y)\in  \cE(\cH)$ and 
$$\|\pi_1(x)\pi_2(y)\|_{\cE(\cH)}\leq C_E \|x\otimes y\|_{E(\mathcal{A}_1\otimes \mathcal{A}_2)}.$$
\end{corollary}

It is well-known that the spaces $L_p(0,\infty)$ and $L_{p,\infty}(0,\infty)$ for $p>2$  are interpolation spaces in the pair $(L_2(0,\infty),L_\infty(0,\infty))$. Therefore, we obtain the following result.

\begin{corollary}
Let $p>2$ and let $x\in S(\cA_1,\tau_1)$, $y\in S(\cA_2,\tau_2)$. 
\begin{enumerate}
\item If $x\otimes y\in \cL_p(\cA_1\otimes \cA_2)$, then, $\pi_1(x)\pi_2(y)\in \cL_p(\cH)$ and 
$$\|\pi_1(x)\pi_2(y)\|_{\cL_{p}(\cH)}\leq 532 \|x\otimes y\|_{\cL_{p}(\mathcal{A}_1\otimes \mathcal{A}_2)}=532\|x\|_{\cL_p(\cA_1)}\|y\|_{\cL_p(\cA_2)}.$$
\item If $x\otimes y\in \cL_{p,\infty}(\cA_1\otimes \cA_2)$, then $\pi_1(x)\pi_2(y)\in \cL_{p,\infty}(\cH)$  and 
$$\|\pi_1(x)\pi_2(y)\|_{\cL_{p,\infty}(\cH)}\leq C_p \|x\otimes y\|_{\cL_{p,\infty}(\mathcal{A}_1\otimes \mathcal{A}_2)}.$$
\end{enumerate} 
\end{corollary}

Consider the particular case, when $\cA_1=L_\infty(\bbR^d)$ and $\cA_2=L_\infty(\bbR^d)$ and the representation $\pi_i$ of $\cA_i$, $i=1,2$, on the Hilbert space $L_2(\bbR^d)$ is given by
$$\pi_1(f)=M_f,\quad \pi_2(g)=g(-i\nabla),$$
where $M_f$ is multiplication operator on $L_2(\bbR^d)$ by $f$ and $\nabla=(\partial_1,\dots, \partial_d)$.
It is well known that Hypothesis \ref{abs_Cwikel_hyp} is satisfied. Therefore, the result of \cite{Cwikel} is a particular case of our abstract Cwikel estimate.

\section{The Cwikel estimates for the classical setting}\label{p<2}

The main result of this section is Theorem \ref{clas_Cwikel_E}. Its immediate consequence, Corollary \ref{clas_Cwikel_p<2}, complements the classical result by Cwikel \cite{Cwikel} for the case $p<2$. The similar result is stated in \cite[Sections 5.7, 5.8]{BKS} which relies on \cite[Theorem 11.1]{BS}. However, the authors of the present paper could not confirm the validity of the proof of the latter result from \cite{BS}. We note that our approach provides an elementary proof of the above mentioned result.

We start with an auxiliary lemma.
\begin{lemma}\label{fourier_series_estimate} Let $h\in L_1(\mathbb{R}^d)$ be supported on $[0,1]^d$ and let $\phi\in C^\infty(\bbR^d)$ be supported on $[-3,3]^d$ such that $\phi|_{[0,1]^d}=1.$ 
For the function $\psi=\phi\cdot \cF^{-1}h$ and every $0<p\leq 2$ we have 
$$\|\{\hat{\psi}(\bk)\}_{\bk\in\bbZ^d}\|_p^p\leq C_p \|h\|_1,$$
where $\{\hat{\psi}(\bk)\}_{\bk\in\bbZ^d}$ denotes the sequence on Fourier coefficients of $\psi$ on the cube $[-\pi,\pi]^d$. 
\end{lemma}

\begin{proof}
Since $h\in L_1(\bbR^d)$ is compactly supported, it follows (see e.g. \cite{Stein_Weiss}) that $\mathcal{F}^{-1}h\in C^{\infty}(\mathbb{R}^d)$, where $\cF$ denotes the Fourier transform on $\bbR^d$. 
By the assumptions, the function $\psi=\phi\cdot\mathcal{F}^{-1}h\in C^{\infty}([-\pi,\pi]^d)$ and vanishes near the boundary together with all its derivatives. 

Let us fix an even integer $m>\frac{d}{p}.$ We have 
\begin{align*}
\sum_{{\bf k}\neq0}|\hat{\psi}(\bk)|^p=\sum_{{\bk}\neq0}|{\bk}|^{-mp}(|{\bk}|^m|\hat{\psi}(\bk)|)^p\leq(\sum_{{\bk}\neq0}|{\bk}|^{-mp})\cdot(\sup_{{\bk}\neq0}|{\bk}|^m|\hat{\psi}(\bk)|)^p.
\end{align*}
Since $\psi\in C^{\infty}([-\pi,\pi]^d)$ and $m$ is even, it follows that $|{\bk}|^m\hat{\psi}(\bk)$ is the ${\bf k}-$th Fourier coefficient of $(-\Delta)^{\frac{m}{2}}(\psi).$ Hence,
$$\sup_{{\bk}\neq0}|{\bk}|^m|\hat{\psi}(\bk)|\leq \|(-\Delta_{\mathbb{T}})^{\frac{m}{2}}\psi\|_{\infty}\leq d^{\frac{m}{2}}\|\psi\|_{C^m([-\pi,\pi]^d)}.$$
By definition of $\psi$ and Leibniz rule, we have
$$\sup_{{\bf k}\neq0}|{\bf k}|^m|\hat{\psi}(\bk)|\leq d^{\frac{m}{2}}2^{m}\|\phi\|_{C^m(\mathbb{R}^d)}\|\mathcal{F}^{-1}h\|_{C^m(\mathbb{R}^d)}.$$
To estimate the norm $\|\mathcal{F}^{-1}h\|_{C^m(\mathbb{R}^d)}$,  
for every ${\bf l}\in\mathbb{Z}_+^d,$ we set $h_{\bf l}(t)=\prod_{j=1}^dt_j^{l_j},$ $t\in\mathbb{R}^d.$
We have,
\begin{align*}
\|\mathcal{F}^{-1}h\|_{C^m(\mathbb{R}^d)}&=\max_{\|{\bf l}\|_1\leq m}\|h_{\bf l}(\nabla)\mathcal{F}^{-1}h\|_{\infty}=\max_{\|{\bf l}\|_1\leq m}\|\mathcal{F}^{-1}(h_{\bf l}h)\|_{\infty}\\
&\leq\max_{\|{\bf l}\|_1\leq m}\|h_{\bf l}h\|_1\leq\|h\|_1,
\end{align*}\
where the last inequality follows from the fact that $h$ is supported on $[0,1]$.

Thus, 
\begin{align*}
\sum_{{\bf k}\neq0}|\hat{\psi}(\bk)|^p\leq \const \Big(\sum_{{\bk}\neq0}|{\bk}|^{-mp}\Big)\cdot \|h\|_1^p=\const \|h\|_1^p,
\end{align*} 
which concludes the proof.
\end{proof}

The following lemma is an auxiliary version of Cwikel estimates for compactly supported functions.

\begin{lemma}\label{classical_Cwikel_compactly_supported}
 Let $f\in L_2(\mathbb{R}^d)$ and $g\in L_2(\mathbb{R}^d)$ be supported on $[0,1]^d$. For every $0<p\leq 2,$ we have that $M_fg(-i\nabla)\in \cL_p(L_2(\bbR^d))$ and
$$\|M_fg(-i\nabla)\|_p\leq \const \|f\|_2\|g\|_2,$$
where the constant depends on $d$ and $p$ only. 
\end{lemma}
\begin{proof} 
It is well-known that (see e.g. \cite[Theorem IX.29]{RS_book_II})
\begin{equation}\label{repr_via_int}
(M_fg^2(-i\nabla)M_f) \xi)(s)=\frac1{(2\pi)^{d/2}}\int_{[0,1]^d} f(s)f(t) (\cF^{-1} g^2)(s-t)\xi(t)dt, \quad \xi\in L_2(\bbR^d).
\end{equation}

 Let $\phi$ be a Schwartz function which is $1$ on $[-1,1]^d$ and which is $0$ outside of $[-3,3]^d.$  We set $\psi:=\phi\cdot \cF^{-1}g^2$. Since $g\in L_2(\bbR^d)$, it follows that $\psi\in C^\infty([-\pi,\pi]^d)$ and vanishes near the boundary together with all its derivatives.  
Hence (see e.g. \cite[Chapter VII]{Stein_Weiss}) we have that for almost every $\bu\in[-\pi,\pi]^d$
\begin{equation}\label{Fourier_series_1}
\psi(\bu)=\sum_{\bk\in\bbZ^d}\hat{\psi}(\bk)e^{i\langle \bk,\bu\rangle}, \quad 
\sum_{\bk\in\bbZ^d}|\hat{\psi}(\bk)|<\infty.
\end{equation}

Since $\phi$ equals $1$ for $u\in [-1,1]^d$, it follows that
$$(\mathcal{F}^{-1}g^2)(u)=\sum_{{\bf k}\in\mathbb{Z}^d}\hat{\psi}(\bk)e^{i\langle {\bf k},u\rangle},\quad u\in[-1,1]^d.$$

We set $f_{\bf k}(t)=f(t)e^{i\langle {\bf k},t\rangle},$ and define the rank-one operator on $L_2(\mathbb{R}^d)$ by setting
$$A_{\bf k}\xi:=\langle \xi,f_{\bf k}\rangle f_{\bf k},\quad \xi\in L_2(\mathbb{R}^d).$$
By \eqref{repr_via_int}  we obtain that
\begin{align*}
(M_fg^2(-i\nabla)M_f \xi)(s)&=\frac{f(s)}{(2\pi)^{d/2}}\int_{[0,1]^d}f(t)\xi(t) \sum_{{\bf k}\in\mathbb{Z}^d}\hat{\psi}(\bk)e^{i\langle {\bf k},s-t\rangle}dt.
\end{align*}
By \eqref{Fourier_series_1} and the dominated convergence theorem we infer that 
\begin{align*}
(M_fg^2(-i\nabla)M_f \xi)(s)&=\frac{1}{(2\pi)^{d/2}}\sum_{{\bf k}\in\bbZ}\hat{\psi}(\bk) f(s)e^{i\langle {\bf k},s\rangle}\int_{[0,1]^d}f(t)\xi(t)e^{-i\langle {\bf k},t\rangle}dt\\
&=\frac{1}{(2\pi)^{d/2}}\sum_{{\bf k}\in\mathbb{Z}^d}\hat{\psi}(\bk)(A_{\bf k}\xi)(s).
\end{align*}
The quasi-norm $\|\cdot\|_{p/2}$ is $\frac{p}2$-subadditive (see e.g. \cite{F-space}), and therefore, we have that
\begin{align}\label{compsup_first_estimate}
\|M_fg(-i\nabla)\|_p^{p}&=\|M_fg^2(-i\nabla)M_f\|_{p/2}^{p/2}\leq\frac{1}{(2\pi)^{d/2}}\sum_{{\bf k}\in\mathbb{Z}^d}|\hat{\psi}(\bk)|^{p/2}\|A_{\bf k}\|_{p/2}^{p/2}\nonumber\\
&=\frac{1}{(2\pi)^{d/2}}\|f\|_2^{p}\sum_{{\bf k}\in\mathbb{Z}^d}|\hat{\psi}(\bk)|^{p/2}.
\end{align}

By Lemma \ref{fourier_series_estimate} (with $h=g^2$) we obtain that 
$$\|M_fg(-i\nabla)\|_p^{p}\leq \const \|f\|_2^{p}\|g^2\|_1^{p/2}=\const \|f\|_2^p\|g\|_2^p,$$
which concludes the proof. 
\end{proof}

\begin{lemma}\label{oper_otimes}
Let $E$ be a symmetric quasi-Banach sequence space with modulus of concavity $M$. Then for every $p>0$ such that $2^{\frac1p-1}>M$,  the operator $T:E(\bbZ_+)\to E(\bbZ_+^2)$ defined by 
$$T:x\mapsto x\otimes \{(k+1)^{-1/p}\},$$ is bounded.
\end{lemma}
\begin{proof}
By the Aoki-Rolewicz theorem, for $q>0$ satisfying $2^{\frac1q-1}=K$ (passing to an equivalent quasi-norm, if necessary) we have 
$$\|x+y\|_E^q\leq \|x\|_E^q+\|y\|_E^q.$$
Hence, for  $p<q$ we have 
$$\|x\otimes \{(k+1)^{-1/p}\}\|_E^q\leq \sum_{n=1}^\infty n^{-q/p} \|x\|_E^q=\const \|x\|_E^q.$$
%Fix $0<p<\min\{p_1,p_2\}$ and let $x\in l_{p_j}$, $j=1,2$.  We have 
%$$\|x\otimes \{(k+1)^{-1/p}\|_{p_j}=\|x\|_{p_j}\| \{(k+1)^{-1/p}\}\|_{p_j}=\|x\|_{p_j}\Big(\sum_{k=0}^\infty \frac1{(k+1)^{p_j/p}}\Big)^{\frac1{p_j}}.$$
%Since $\frac{p_j}{p}>1$ it follows that $T$ is a bounded operator in $l_{p_j}$, $j=1,2$. Since $E$ is an interpolation space in $(l_{p_1}, l_{p_2})$, the claim follows. 
\end{proof}

The following theorem is the main result of this section. It is stated is somewhat abstract manner. For the case of Schatten and weak Schatten ideals the statement is given explicitly in Corollary \ref{clas_Cwikel_p<2} below.
To formulate the result we introduce firstly the definition of space $E(L_2)(\bbR^d)$. For the special case, when $E=\ell_p$ or $E=\ell_{p,\infty}$ this definition coincides with that given in \cite{BS} (see also \cite[Chapter 4]{Simon}).
\begin{definition}Let $K$ be the unit cube in $\bbR^d$ and let  $E\subset \ell_2$ be a symmetric quasi-Banach sequence space. Let $E(L_2)(\bbR^d)$ be the space of all (measurable) functions such that the sequence $\{\|f\chi_{K+\bm}\|_2\}_{\bm\in\bbZ^d}$ is in $E$. For $f\in E(L_2)(\bbR^d)$ set 
$$\|f\|_{E(L_2)(\bbR^d)}=\Big\|\{\|f\chi_{K+\bm}\|_2\}_{\bm\in\bbZ^d}\Big\|_E.$$
The spaces $E(L_q)(\bbR^d)$ for $0<q\leq \infty$ are defined similarly.
\end{definition}

\begin{theorem}\label{clas_Cwikel_E}
Let $E\subset \ell_2$ be a symmetric quasi-Banach sequence space such that $\mu^2(x)\prec \mu^2(y),\, x\in E$ implies that $y\in E$ and  $\|y\|_E\leq c_E\|x\|_E$. If $f\otimes g\in E(L_2)(\bbR^{2d})$, then $M_fg(-i\nabla)\in \cE(L_2(\bbR^d))$ and 
$$\|M_fg(-i\nabla)\|_{\cE(L_2(\bbR^d))}\leq c_{E,p} \|f\otimes g\|_{E(L_2)(\mathbb{R}^{2d})}.$$
\end{theorem}

\begin{proof}
Let $K$ be the unit cube in $\bbR^d$. We denote
$$f_\bm=f\chi_{\bm+K},\quad g_\bm=g\chi_{\bm+K},\quad \bm\in\bbZ^d.$$

Since $E\subset \ell_2$ we have that 
$$\sum_{\bm_1,\bm_2\in\bbZ^d}\|M_{f_{\bm_1}}g_{\bm_2}(-i\nabla)\|_2^2=\sum_{\bm_1,\bm_2\in\bbZ^d}\|f_{\bm_1}\otimes g_{\bm_2}\|_2^2=\|f\otimes g\|_2^2<\infty.$$
Hence, the assumptions of Lemma \ref{lem_strong_major} are satisfied, and therefore
\begin{align}\label{clas_Cwikel_via_direct}
\mu^2\Big(\bigoplus_{\bm_1,\bm_2\in\bbZ^d} &M_{f_{\bm_1}}g_{\bm_2}(-i\nabla)\Big)\nonumber\\
&\prec \mu^2\Big(\sum_{\bm_1,\bm_2\in\bbZ^d}M_{f_{\bm_1}}g_{\bm_2}(-i\nabla)\Big)=\mu^2(M_fg(-i\nabla)).
\end{align}

We claim that
\begin{equation*}\label{clas_Cwikel_direct_sum}
\bigoplus_{\bm_1,\bm_2\in\bbZ^d} M_{f_{\bm_1}}g_{\bm_2}(-i\nabla)\in \cE(L_2(\bbR^d)).
\end{equation*}

The functions $f(\cdot-\bm_1)\chi_K$ and $g(\cdot-\bm_2)\chi_K$ are supported on $[0,1]^d$ and by the  assumptions $f(\cdot-\bm_1)\chi_K,g(\cdot-\bm_2)\chi_K\in L_2(\bbR^d)$. Therefore Lemma \ref{classical_Cwikel_compactly_supported} implies that for $p>0$ with $2^{\frac1p-1}>M$ we have 
\begin{align*}
\|M_{f_{\bm_1}}g_{\bm_2}(-i\nabla)\|_p\leq c_p\|f(\cdot-\bm_1)\chi_K\|_2\|g(\cdot-\bm_2)\chi_K\|_2=c_p\|f_{\bm_1}\otimes g_{\bm_2}\|_2.
\end{align*}
Recalling that $\|\cdot\|_{p,\infty}\leq \|\cdot\|_p$, we obtain 
$$\mu(M_{f_{\bm_1}}g_{\bm_2}(-i\nabla))\leq \|M_{f_{\bm_1}}g_{\bm_2}(-i\nabla)\|_{p,\infty}\cdot z_p\leq c_p\|f_{\bm_1}\otimes g_{\bm_2}\|_2\cdot z_p,$$
where $z_p=\{\frac1{(k+1)^{1/p}}\}_{k\in \bbN}.$

Hence, (everywhere below the summation is taken over $\bbZ^d$)
\begin{align*}
\Big\|\bigoplus_{\bm_1,\bm_2}& M_{f_{\bm_1}}g_{\bm_2}(-i\nabla)\Big\|_{\cE}=\Big\|\bigoplus_{\bm_1,\bm_2} \mu(M_{f_{\bm_1}}g_{\bm_2}(-i\nabla))\Big\|_{E}\\
&\leq \Big\|\bigoplus_{\bm_1,\bm_2}c_p\|f_{\bm_1}\otimes g_{\bm_2}\|_2\cdot z_p\Big\|_E=c_{p}\Big\|\{\|f_{\bm_1}\|_2\|g_{\bm_2}\|_2\}_{\bm_1,\bm_2}\otimes z_p\Big\|_E\\
&=c_p\Big\|T(\{\|f_{\bm_1}\|_2\|g_{\bm_2}\|_2\}_{\bm_1,\bm_2})\Big\|_E.
\end{align*}
By Lemma \ref{oper_otimes} we infer that 
\begin{align}\label{clas_Cwikel_direct_sum_norm}
\Big\|\bigoplus_{\bm_1,\bm_2\in\bbZ^d} M_{f_{\bm_1}}g_{\bm_2}(-i\nabla)\Big\|_{\cE}&\leq c_{E,p}\|\{\|f_{\bm_1}\otimes g_{\bm_2}\|_2\}_{\bm_1,\bm_2}\|_E\nonumber\\
&=c_{E,p}\|f\otimes g\|_{E(L_2)(\mathbb{R}^{2d})}<\infty.
\end{align}

Thus, $\bigoplus_{\bm_1,\bm_2} M_{f_{\bm_1}}g_{\bm_2}(-i\nabla)\in \cE(L_2(\bbR^d))$, and therefore \eqref{clas_Cwikel_via_direct} and the assumption on the space $E$ (with $x=\bigoplus_{\bm_1,\bm_2} M_{f_{\bm_1}}g_{\bm_2}(-i\nabla)$ and $y=M_fg(-i\nabla)$) implies that $M_fg(-i\nabla)\in \cE(L_2(\bbR^d))$ and  
\begin{align*}
\|M_fg(-i\nabla)\|_{\cE(L_2(\bbR^d))}&\leq c_{E}\Big\|\bigoplus_{\bm_1,\bm_2} M_{f_{\bm_1}}g_{\bm_2}(-i\nabla)\Big\|_\cE\\
&\stackrel{\eqref{clas_Cwikel_direct_sum_norm}}{\leq} c_{E,p}\|f\otimes g\|_{E(L_2)(\mathbb{R}^{2d})},
\end{align*}
as required.
 
\end{proof}

Proposition \ref{clas_Lp_weakLp} implies that the sequence spaces $\ell_p$, $\ell_{p,\infty}$ satisfy the assumption of Theorem \ref{clas_Cwikel_p<2}. Hence, as an immediate consequence of Theorem  \ref{clas_Cwikel_p<2} and Proposition \ref{clas_Lp_weakLp}  we obtain Cwikel estimates for the classical setting for $p<2$ (see \cite{BKS,BS,Simon}). We note, that our approach provides an elementary and direct proof of the estimates without drawing fine estimates from the interpolation theory.

\begin{corollary}\label{clas_Cwikel_p<2}
Let $0<p<2$.
\begin{enumerate}
\item  If $f\otimes g\in \ell_p(L_2)(\bbR^{2d})$, then $M_fg(-i\nabla)\in \cL_p(L_2(\bbR^d))$ and
$$\|M_fg(-i\nabla)\|_p\leq C_p \|f\otimes g\|_{\ell_{p}(L_2)(\bbR^{2d})}.$$
\item If $f\otimes g\in \ell_{p,\infty}(L_2)(\bbR^{2d})$, then $M_fg(-i\nabla)\in \cL_{p,\infty}(L_2(\bbR^d))$ and
$$\|M_fg(-i\nabla)\|_{p,\infty}\leq C_p \|f\otimes g\|_{\ell_{p,\infty}(L_2)(\bbR^{2d})}.$$
\end{enumerate} 
\end{corollary}
\section{Cwikel estimates for the classical setting for weak  Schatten ideal}

In this section we show that Cwikel estimate does not hold for $\mathcal{L}_{2,\infty}.$

\begin{theorem}\label{Cwikel_L2,infty_contexample}
 There exists $f\in L_2(\mathbb{R})$ such that $M_f(1-\Delta)^{-\frac14}\notin\mathcal{L}_4(L_2(\bbR)).$ In particular, $M_f(1-\Delta)^{-\frac14}\notin\mathcal{L}_{2,\infty}(L_2(\bbR)).$
\end{theorem}

\begin{proof} 
Let $f(t):=t^{-\frac12}|\log(t)|^{-1}\chi_{(0,\frac12)}(t),$ $t\in\mathbb{R}.$ It is clear that $f\in L_2(\bbR)$. We claim that
$$M_f(1-\Delta)^{-\frac12}M_f\notin\mathcal{L}_2(L_2(\bbR)).$$
Since the integral kernel of the operator $(1-\Delta)^{-\frac12}$ can be written in terms of the Macdonald function $(t,s)\to K_0(|t-s|),\quad t,s\in\mathbb{R}$ (see e.g. \cite[(9.25)]{Abr_Stegun}), it follows that the integral kernel of the operator $M_f(1-\Delta)^{-\frac12}M_f$ is given by the formula
$$(t,s)\to f(t)f(s)K_0(|t-s|),\quad t,s\in\mathbb{R}.$$
If $t,s\in(0,\frac12),$ then $|t-s|<1,$ and therefore $K_0(|t-s|)\approx \log(|t-s|)$ (see \cite[9.6.21]{Abr_Stegun}).

We have 
\begin{align*}
\int_{(0,\frac12)^2}\frac{\log^2(|t-s|)}{ts\cdot\log^2(t)\cdot\log^2(s)}dtds&\geq
\int\limits_{\substack{t\geq 2s\\t,s\in (0,\frac12)}}\frac{\log^2(|t-s|)}{ts\cdot\log^2(t)\cdot\log^2(s)}dtds 
\\&\geq \int\limits_{\substack{|t|\geq 2|s|\\t,s\in (0,\frac12)}}\frac{1}{ts\cdot\log^2(t)}dtds.
\end{align*}
Since the latter integral diverges, we infer that the integral kernel of the operator $M_f(1-\Delta)^{-\frac12}M_f$ does not belong to $L_2(\mathbb{R}\times\mathbb{R}).$  Therefore, the operator $M_f(1-\Delta)^{-\frac12}M_f$ is not in $\cL_2(L_2(\bbR))$, which implies that $M_f(1-\Delta)^{-\frac14}\notin \cL_4(L_2(\bbR))$. 
\end{proof}

In the rest of this section we establish a positive result for Cwikel estimates in $\cL_{2,\infty}(L_2(\bbR^d))$.
The crucially important property here is the logarithmic convexity of $\cL_{1,\infty}(\cM).$ It was established by Stein and Weiss \cite{Stein_Weiss_logconvexity} and by Kalton \cite{Kalton_logconvex} in the commutative situation. We prove this property for the noncommutative counterpart $\cL_{1,\infty}(\cM)$, where $\cM$ is a semifinite von Neuman algebra equipped with a faithful normal semifinite trace $\tau$. We present the proof of Lemma \ref{kalton lagrange} below for convenience of the reader. Its proof is taken from \cite{Kalton_logconvex}. Lemma \ref{logconvex} below is inspired by  \cite[Theorem 3.4]{Kalton_logconvex}, but the latter is essentially commutative and so, our proof is quite different.

\begin{lemma}\label{kalton lagrange} Let $a_k\in[0,1],$ $1\leq k\leq n,$ be such that $\sum_{k=1}^na_k=1.$ It follows that
\begin{equation}\label{log_convexity_com}
\sum_{k=1}^na_k\log(\frac{e}{a_k})\leq 2\sum_{k=1}^na_k(1+\log(k)).
\end{equation}
\end{lemma}
\begin{proof} Consider the function
$$f:(a_1,\cdots,a_n)\to \sum_{k=1}^na_k\log(\frac{e}{a_k})-2\sum_{k=1}^na_k\log(k)$$
on the simplex $\{(a_1,\dots, a_n): a_k\in[0,1],\, k=1,\dots n,\, \sum_{k=1}^na_k=1\}.$. Since $f$ is continuous, it attains its supremum on the simplex.

Suppose first, that supremum belongs to the interior of the simplex. We can find it by Lagrange method. Set $F=\sum_{k=1}^na_k.$ We have $\nabla(f-\lambda F)=0$ at the maximum point, where $\lambda$ is the Lagrange multiplier. Computing the partial derivatives, we obtain that  
$$\log(\frac1{a_k})-2\log(k)=\lambda,$$
and therefore, the stationary point of $f$ is given by
 $$\tilde{a}_k=\frac1{e^{\lambda}k^2},\quad 1\leq k\leq n.$$
Since $\sum_{k=1}^na_k=1,$ it follows that $e^{\lambda}=\sum_{k=1}^n\frac1{k^2}\leq e$ and, therefore, $\lambda\leq 1.$ 
Hence, computing the value of the function $f$ at the stationary point $(\tilde{a}_1,\dots,\tilde{a}_n)$ we have 
\begin{align*}
f(\tilde{a}_1,\dots,\tilde{a}_k)&=\sum_{k=1}^n\tilde{a}_k\log(\frac{e}{\tilde{a}_k})-2\sum_{k=1}^n\tilde{a}_k\log(k)\\
&=1+\sum_{k=1}^n \frac1{e^\lambda k^2}\log(e^\lambda k^2)-2\sum_{k=1}^n\frac1{e^\lambda k^2}\log(k)\\
&=1+\lambda \sum_{k=1}^n \frac1{e^\lambda k^2}=(1+\lambda)\leq 2.
\end{align*}This shows that inequality \eqref{log_convexity_com} holds in this case. 

Suppose now that supremum is attained on the boundary of the simplex, that is, $a_m=0,$ for some $1\leq m\leq n.$ Set $b_k=a_k,$ $1\leq k<m,$ and $b_k=a_{k+1},$ $m\leq k<n.$ By induction on $n,$ we have
$$\sum_{k=1}^na_k\log(\frac{e}{a_k})=\sum_{k=1}^{n-1}b_k\log(\frac{e}{b_k})\leq2\sum_{k=1}^{n-1}b_k(1+\log(k))\leq 2\sum_{k=1}^na_k(1+\log(k)),$$
which concludes the proof.
\end{proof}

\begin{lemma}\label{logconvex} If $x_k\in \cL_{1,\infty}(\cM),$ $1\leq k\leq n,$ then
$$\Big\|\sum_{k=1}^n x_k\Big\|_{1,\infty}\leq 4\sum_{k=1}^n\|x_k\|_{1,\infty}(1+\log(k)).$$
\end{lemma}
\begin{proof} Let $x_k\in \cL_{1,\infty}(\cM),$ $1\leq k\leq n,$ and let $x=\sum_{k=1}^nx_k.$ Suppose that
\begin{equation}\label{weak_L2_prel_assum}
\sum_{k=1}^n\|x_k\|_{1,\infty}=1.
\end{equation}

Set $p_k=E_{|x_k|}(\frac1t,\infty)$ and $p=\bigvee_{1\leq k\leq n}p_k.$ Since $x_k\in \cL_{1,\infty}(\cM)$, we have $\tau(p_k)\leq t\|x_k\|_{1,\infty}$ and, therefore, 
$$\tau(p)=\tau\Big(\bigvee_{1\leq k\leq n}p_k\Big)\leq \sum_{k=1}^n \tau(p_k)\leq  t\sum_{k=1}^n \|x_k\|_{1,\infty}=t.$$

Since $\mu(t,xp)=0,$ equality \eqref{mu_of_sum}, combined with the fact that singular value function is decreasing, implies that
$$\mu(2t,x)\leq\mu(t,xp)+\mu(t,x(1-p))\leq\frac1t\int_0^t\mu(s,x(1-p))ds.$$
It follows from \eqref{submaj_sum} that
\begin{align}\label{weak_L2_muA_into_sum}
t\mu(2t,x)&\leq \int_0^t\mu(s,x(1-p))ds\stackrel{\eqref{submaj_sum}}{\leq} \sum_{k=1}^n\int_0^t\mu(s,x_k(1-p))ds\nonumber\\
&\leq\sum_{k=1}^n\int_0^t\mu(s,x_k(1-p_k))ds.
\end{align}
Due to the choice of  $p_k,$ for every fixed $s>0$ we have that 
$$\mu(s,x_k(1-p_k))\leq\min\{\frac1t,\mu(s,x_k)\}\leq\min\{\frac1t,\frac{\|x_k\|_{1,\infty}}{s}\},\quad s>0.$$
Hence, 
\begin{align*}
\int_0^t\mu(s,x_k(1-p_k))ds=\int_0^{t\|x_k\|_{1,\infty}}\frac1t ds+\int_{t\|x_k\|_{1,\infty}}^t\frac{\|x_k\|_{1,\infty}}{s}ds\\
=\|x_k\|_{1,\infty}-\|x_k\|_{1,\infty} \log(\|x_k\|_{1,\infty})=\|x_k\|_{1,\infty}\log\Big(\frac{e}{\|x_k\|_{1,\infty}}\Big).
\end{align*}
Substituting to \eqref{weak_L2_muA_into_sum}, we infer
\begin{align*}
t\mu(2t,x)&\leq\sum_{k=1}^n\|x_k\|_{1,\infty}\log\Big(\frac{e}{\|x_k\|_{1,\infty}}\Big).
\end{align*}
Thus, 
$$\Big\|\sum_{k=1}^nx_k\Big\|_{1,\infty}=\sup_{t>0} t\mu(t,x)\leq 2\sum_{k=1}^n\|x_k\|_{1,\infty}\log\Big(\frac{e}{\|x_k\|_{1,\infty}}\Big).$$
Using Lemma \ref{kalton lagrange}, we infer that
$$\Big\|\sum_{k=1}^nx_k\Big\|_{1,\infty}\leq 4\sum_{k=1}^n\|x_k\|_{1,\infty}(1+\log(k)).$$
To conclude the proof one can remove assumption \eqref{weak_L2_prel_assum} by homogeneity.
\end{proof}

The following proposition gives logarithmic convexity of $\cL_{1,\infty}(\cM)$ , which is crucial for Theorem \ref{weak_L2_main} below.
\begin{proposition}\label{log_convexity} If $x_k\in \cL_{1,\infty}(\cM),$ $k\geq 1,$ then
$$\Big\|\sum_{k=1}^{\infty}x_k\Big\|_{1,\infty}\leq 4\sum_{k=1}^{\infty}\|x_k\|_{1,\infty}(1+\log(k)).$$
Here, the series in the left hand side converges in $\cL_{1,\infty}(\cM)$ provided that the series in the right hand side converges.
\end{proposition}
\begin{proof} By Lemma \ref{logconvex}, we have
$$\Big\|\sum_{k=n}^{m-1}x_k\Big\|_{1,\infty}=\Big\|\sum_{k=1}^{m-n}x_{k+n-1}\Big\|_{1,\infty}\leq 4\sum_{k=1}^{m-n}\|x_{k+n-1}\|_{1,\infty}(1+\log(k))\leq$$
$$\leq 4\sum_{k=n}^{m-1}\|x_k\|_{1,\infty}(1+\log(k)).$$
By the assumption, the series  $\sum_{k=1}^{\infty}\|x_k\|_{1,\infty}(1+\log(k))$ converges, and therefore the sequence of  partial sums of $\sum_{k=1}^{\infty}x_k$  is a Cauchy sequence in $\cL_{1,\infty}(\cM).$ Since the space $(\cL_{1,\infty}(\cM),\|\cdot\|_{1,\infty})$ is complete, we conclude that the series $\sum_{k=1}^{\infty}x_k$ converges in $\cL_{1,\infty}(\cM)$ and obtain the required estimate.
\end{proof}

We now introduce an auxiliary class of functions on $\bbR^d$, for which we establish our positive result for $\cL_{2,\infty}(L_2(\bbR^d))$. The definition below is modelled after the definition of spaces $\ell_{p,\infty}(L_2)(\bbR^d)$ in \cite{BS} (see also \cite[Chapter 4]{Simon}).
\begin{definition}\label{def_l_2log}
Let $K$ be the unit cube in $\bbR^d$. Let $\ell_{2,\log}(L_\infty)(\bbR^d)$ be the space of all (measurable) functions on $\bbR^d$, such that 
the sequence 
$\{\|f\chi_{K+\bm}\|_\infty\}_{\bm\in\bbZ^d}$ is square-summable with the weight $(1+\log|\bm|)$, where $|\bm|$ is the Euclidean norm of $\bm\in\bbZ^d$. For $f\in \ell_{2,\log}(L_\infty)(\bbR^d)$ we set 
$$\|f\|_{2,\log}=\Big(\sum_{\bm\in\bbZ^d}(1+\log|\bm|)\|f\chi_{K+\bm}\|_\infty^2\Big)^{1/2}.$$
\end{definition}

The following theorem is the second main result of the present section. It proves a version of Cwikel estimates for the ideal $\cL_{2,\infty}(L_2(\bbR^d))$, the only case in the scale of weak ideals which is covered by neither \cite{Cwikel},\cite{Simon}, nor \cite{BKS},\cite{BS}.
\begin{theorem}\label{weak_L2_main}
If $g\in \ell_{2,\infty}(L_4)(\bbR^d)$ and 
$f\in \ell_{2,\log}(L_\infty)(\bbR^d)$, then $M_fg(-i\nabla)\in \cL_{2,\infty}(L_2(\bbR^d))$ and 
$$\|M_fg(-i\nabla)\|_{2,\infty}\leq \const \|f\|_{2,\log} \|g\|_{\ell_{2,\infty}(L_4)}.$$
\end{theorem}
\begin{proof}We denote by $K$ the unit cube, and set $f_\bm=f\chi_{K+\bm},\bm\in\bbZ^d$. 
For every fixed $\bm\in\bbZ^d$ we have $f_\bm\in L_\infty(\bbR^d)$ and $g^2\in \ell_{1,\infty}(L_2)(\bbR^d)$, and therefore  $f_\bm\otimes g^2\in \ell_{1,\infty}(L_2)(\bbR^d)$ for every $\bm\in\bbZ^d$, and therefore, Corollary \ref{clas_Cwikel_p<2} (ii) implies that $M_{f_\bm}g^2(-i\nabla)\in \cL_{1,\infty}(L_2(\bbR^d))$ and 
$$\|M_{f_\bm}g^2(-i\nabla)\|_{1,\infty}\leq\const  \|f_\bm\|_2\|g^2\|_{\ell_{1,\infty}(L_2)}\leq \|f_\bm\|_\infty\|g\|^2_{\ell_{2,\infty}(L_4)}.$$
Hence, $M_{f_\bm}g^2(-i\nabla)M_{f_\bm}\in \cL_{1,\infty}(L_2(\bbR^d))$ and
$$\|M_{f_\bm}g^2(-i\nabla)M_{f_\bm}\|_{1,\infty}\leq \|f_\bm\|^2_\infty\|g\|^2_{\ell_{2,\infty}(L_4)}.$$

By \eqref{mu^2_mu} we obtain that $g(-i\nabla)M_{f_{\bm}^2}g(-i\nabla)\in\cL_{1,\infty}(L_2(\bbR^d))$ and 
$$\|g(-i\nabla)M_{f_{\bm}^2}g(-i\nabla)\|_{1,\infty}\leq \const \|f_\bm\|_\infty^2\|g\|^2_{l_{2,\infty}(L_4)}.$$

We have 
\begin{align*}
\sum_\bm(1+\log |\bm|)&\|g(-i\nabla)M_{f_{\bm}^2}g(-i\nabla)\Big\|_{1,\infty}\\
&\leq \const \sum_\bm(1+\log |\bm|)\|f\chi_{K+\bm}\|^2_\infty\|g\|^2_{l_{2,\infty}(L_4)}\\
&=\const \|f\|^2_{2,\log}\|g\|^2_{l_{2,\infty}(L_4)}<\infty.
\end{align*}
Combining preceding estimate with Proposition \ref{log_convexity}, we have that
\begin{align*}
\|M_fg(-i\nabla)\|_{2,\infty}^2&=\|g(-i\nabla)M_{f^2}g(-i\nabla)\|_{1,\infty}=\Big\|\sum_{\bm}g(-i\nabla)M_{f_{\bm}^2}g(-i\nabla)\Big\|_{1,\infty}\\
&\leq 4\sum_\bm(1+\log |\bm|)\|g(-i\nabla)M_{f_{\bm}^2}g(-i\nabla)\Big\|_{1,\infty}\\
&\leq \const \|f\|^2_{2,\log}\|g\|^2_{l_{2,\infty}(L_4)},
\end{align*}
which concludes the proof.
\end{proof}

\section{Noncommutative Euclidean space}\label{sec_nc_plane}
The noncommutative Euclidean space, also known as Moyal plane admits several equivalent definitions. For example,
\begin{enumerate}
    \item Define a symmetric bilinear product $\star_\theta$ on $\cS(\bbR^d)$. The algebra of smooth functions is then defined to be the algebra $(\cS(\bbR^d),\star_\theta)$ (see e.g. \cite{Gayral_Iochum_Varilly}).
    \item Define $\bbR^d_\theta$ as an isospectral deformation of the manifold $\bbR^d$ with respect to a certain $\bbR^d$-action (see e.g. \cite{Rieffel}).
    \item Define firstly $L^{\infty}(\bbR^d_\theta)$ as twisted (reduced) group von Neumann algebra  and then introduce `function spaces' on $\bbR^d_\theta$ as noncommutative spaces associated to that algebra.
\end{enumerate}

In this paper, we stick with the third option. Our approach is unitarily equivalent, via the Fourier transform, to the standard Moyal product approach in (1), and therefore, the content of this section is well known to the experts. For convenience, we give a self contained 
exposition with references to the equivalent results based on Moyal product, when possible. For simplicity of exposition we totally avoid terminology from twisted group algebras.

\subsection{Definition and elementary properties of noncommutative Euclidean  space}
We shall define noncommutative Euclidean space as the von Neumann algebra generated by a strongly continuous family of unitary operators $\{U(\bt)\}_{\bt\in\bbR^d}$, $d\in\bbN$, satisfying the commutation relation 
\begin{equation}\label{nc_plane_comm_relation}
U(\bt+\bs) = \exp(-\frac{i}{2}\langle \bt,\theta \bs\rangle) U(\bt)U(\bs),\quad \bt,\bs \in \bbR^d,
\end{equation}
where $\theta$ is a fixed antisymmetric real $d\times d$ matrix. Namely, we set 
\begin{equation}\label{nc_plane_realisation}
(U(\bt)\xi)(\bu)=e^{-\frac{i}{2}\langle \bt,\theta \bu\rangle}\xi(\bu-\bt),\quad \xi\in L_2(\mathbb{R}^d), \quad \bu,\bt\in \bbR^d.
\end{equation}
It is clear that the family $\{U(\bs)\}_{\bs\in\bbR^d}$ of unitaries is strongly continuous. Moreover, for every $\bs,\bt\in\bbR^d$ we have 
\begin{align*}
(U(\bs)U(\bt)\xi)(\bu)&=e^{-\frac{i}{2}\langle \bs,\theta \bu\rangle}e^{-\frac{i}{2}\langle \bt,\theta (\bu-\bs)\rangle}\xi(\bu-\bt-\bs)\\
&=e^{\frac{i}{2}\langle \bt,\theta \bs\rangle}e^{-\frac{i}{2}\langle \bs+\bt,\theta \bu\rangle}\xi(\bu-\bt-\bs)\\
&=e^{\frac{i}{2}\langle \bt,\theta \bs\rangle}(U(\bs+\bt)\xi)(\bu).
\end{align*}
In other words, the operators defined in \eqref{nc_plane_realisation} satisfy the commutation relation \eqref{nc_plane_comm_relation}. 

\begin{definition}
\label{nc_plane_definition}
Let $d\in\bbN$ and let $\theta$ be a fixed antisymmetric real $d\times d$ matrix. 
The von Neumann algebra on $L_2(\bbR^d)$ generated by $\{U(\bt)\}_{\bt \in \bbR^d}$, introduced in \eqref{nc_plane_realisation}, is said to be noncommutative Euclidean  space and denoted by $L_\infty(\bbR^d_\theta)$.
\end{definition}

\begin{remark}

\begin{enumerate}

\item The classical case of Euclidean  space is recovered by taking $\theta = 0$. Definition \ref{nc_plane_definition} states that $L_\infty(\bbR^d_0)$
    is the von Neumann algebra algebra generated by $d$ commuting unitary groups, and this is $*$ isomorphic to $L^{\infty}(\bbR^d)$.
    
\item  Note that relation \eqref{nc_plane_comm_relation} makes no sense without the assumption that $\theta$ is antisymmetric and real.
Indeed, we have
    \begin{equation*}
        \exp(-\frac{i}{2}\langle\bt,\theta \bs\rangle) = U(\bt+\bs)U(\bs)^*U(\bt)^*.
    \end{equation*}
    Since the right hand side of the above is unitary, we require that $\langle\bt,\theta \bs\rangle$ is always real, and so $\theta$ must be real. Next, combining the fact that,
    \begin{align*}
        U(\bt+\bs) &= \exp(-\frac{i}{2}\langle \bt,\theta \bs\rangle)U(\bt)U(\bs)\\
        U(-\bs-\bt) &= \exp(-\frac{i}{2}\langle \bs,\theta \bt\rangle)U(-\bs)U(-\bt),
    \end{align*}
with  the equality $U(\bt)^{-1}=U(-\bt)$, it follows from multiplying the above two equations that
    \begin{equation*}
        \langle \bt,\theta \bs\rangle+\langle \bs,\theta \bt\rangle = 0.
    \end{equation*} Therefore, $\theta$ has to be antisymmetric.\qed
    \end{enumerate}
\end{remark}

Our first step in studying the noncommutative Euclidean  space $L_\infty(\bbR^d)$ is to show a simple and well-known fact (see e.g. \cite[Proposition 5.2]{Rieffel}) that we can replace our original family $\{U(\bt)\}_{\bt\in\bbR^d}$ and the matrix $\theta$, so that the commutation relation \eqref{nc_plane_comm_relation} becomes simpler.
%
%{\color{red}
%\begin{proposition}\label{nc_plane_independence_of_theta}
%Let $\tilde{\theta}$ be a matrix equivalent to $\theta$, that is $\tilde{\theta}=P^{-1}\theta P$ for some invertible matrix $P$. Then $L_\infty(\bbR^d_\theta)=L_\infty(\bbR^d_{\tilde{\theta}})$. 
%\end{proposition}
%\begin{proof}
%We define 
%$$V(\bt)=U(P\bt), \quad \bt\in\bbR^d.$$
%It is immediate that the family of unitaries $\{V(\bt)\}_{\bt\in\bbR^d}$ is strongly continuous. 
%We have 
%\begin{align*}
%V(\bt+\bs) &= U(P\bt+P\bs)\\
%&= \exp(-\frac{i}{2}\langle P\bt,\theta P\bs\rangle)U(QN\bt)U(QN\bs)\\
%&= \exp(-\frac{i}{2}\langle N\bt,\theta_0N\bs\rangle)U(QN\bt)U(QN\bs)\\
%&= \exp(-\frac{i}{2}\langle \bt,\tilde{\theta}\bs\rangle)V(\bt)V(\bs).               
%    \end{align*}
%    So $\{V(\bt)\}_{\bt \in \bbR^d}$ satisfies the required relation. Since both $Q$ and $N$ are invertible, we  have that $U(\bt)=V(N^{-1}Q^{-1}\bt)$, and therefore the family $\{V(\bt)\}_{\bt \in \bbR^d}$ generates $L_\infty(\bbR^d_\theta)$. 
%\end{proof}

\begin{proposition}\label{nc_plane_to_block_diagonal}
Let $k=\dim\ker \theta$ and let $\tilde{\theta}$ be the block-diagonal matrix of the form 
 \begin{equation}\label{nc_plane_theta_block-diagonal}
        \tilde{\theta} = \begin{pmatrix}
                     0 & -1 &\\
                     1 &  0 & \\
                       &    & 0 & -1\\
                       &    & 1 & 0\\
                       &    &   &   & \ddots\\
                       &    &   &   &       &  0 & -1\\
                       &    &   &   &       &  1 & 0\\
                       &    &   &   &       &    &        & 0_{d-2k}
        \end{pmatrix}.
    \end{equation}
 Then, the algebras $L_\infty(\bbR^d_\theta)$ and $L_\infty(\bbR^d_{\tilde{\theta}})$ are spatially isomorphic.
%    
%  There exists a strongly continuous family of unitaries $\{V(\bt)\}_{\bt\in\bbR^d}$ satisfying relation 
%$$V(\bt+\bs) = e^{\frac{i}{2}\langle \bt,\tilde{\theta}\bs\rangle}V(\bt)V(\bs),$$ such that $L_\infty(\bbR^d_\theta)$ is generated by $\{V(\bt)\}_{\bt\in\bbR^d}$. 
\end{proposition}
 \begin{proof}
 Since the matrix $\theta$ is antisymmetric and real, it follows from the spectral theory that there exists an orthogonal matrix $Q$ such that 
$\theta=Q\theta_0Q^{-1}$, where 
\begin{equation*}
      \theta_0 = \begin{pmatrix}
                     0 & -\theta_1 &\\
                     \theta_1 &  0 & \\
                       &    & 0 & -\theta_2\\
                       &    & \theta_2 & 0\\
                       &    &   &   & \ddots\\
                       &    &   &   &       &  0 & -\theta_{k}\\
                       &    &   &   &       &  \theta_{k} & 0\\
                       &    &   &   &       &    &        & 0_{d-2k}
        \end{pmatrix},
    \end{equation*}
    with $\theta_k>0$. Define also the matrix
    \begin{equation*}
        N = \begin{pmatrix} \theta_1^{-1/2}\\
                            & \theta_1^{-1/2}\\
                            & & \theta_2^{-1/2}\\
                            & & &\theta_2^{-1/2}\\
                            & & & & \ddots\\
                            & & & & & 1_{d-2k}
                            \end{pmatrix}.
    \end{equation*}
% We have that 
% $$\theta=QN^{-1}\tilde{\theta}N^{-1}Q^{-1}.$$
 
Let $\{U^{\theta}(\bt)\}_{\bt\in\bbR^d}$ be the family corresponding to $\theta$ and let $\{U^{\tilde{\theta}}(\bt)\}_{\bt\in\bbR^d}$ be the family corresponding to $\tilde{\theta}$. Introducing the unitary operator $W$ on $L_2(\bbR^d)$ by setting 
 \begin{equation}\label{nc_plane_change_coordinaes}
 (W\xi)(\bt)=\det(N)^{-1/2}\cdot\xi(N^{-1}Q^{-1}\bt),\quad \bt\in\bbR^d,
 \end{equation}
 one can obtain that 
 $$W^*U^\theta(\bt)W=U^{\tilde{\theta}}(N^{-1}Q^{-1}\bt),\quad \bt\in\bbR^d,$$
 which implies the claim.
% that is 
%  \begin{align*}
% (U^\theta(\bt)\xi)(\bu)&=e^{-\frac{i}{2}\langle \bt,\theta \bu\rangle}\xi(\bu-\bt),\\
%  (U^{\tilde{\theta}}(\bt)\xi)(\bu)&=e^{-\frac{i}{2}\langle \bt,\tilde{\theta} \bu\rangle}\xi(\bu-\bt),
% \end{align*}
% where $\xi\in L_2(\mathbb{R}^d), \bu,\bt\in \bbR^d.$
 
% Let us also introduce a unitary operator $W$ on $L_2(\bbR^d)$ by setting 
% \begin{equation}\label{nc_plane_change_coordinaes}
% (W\xi)(\bt)=\det(N)^{-1/2}\cdot\xi(N^{-1}Q^{-1}\bt),\quad \bt\in\bbR^d.
% \end{equation}
% Observe that 
% $$ (W^*\xi)(\bt)=\det(N)^{1/2}\cdot\xi(QN\bt),\quad \bt\in\bbR^d.$$
%
%For $\xi\in L_2(\mathbb{R}^d)$ and $\bu,\bt\in \bbR^d$ we have 
%$$(U^\theta(\bt)W\xi)(\bu)=\det(N)^{1/2}e^{-\frac{i}{2}\langle \bt,\theta \bu\rangle}\xi(N^{-1}Q^{-1}(\bu-\bt)).$$
%Therefore, 
%\begin{align*}
%(W^*U^\theta(\bt)W\xi)(\bu)&=e^{-\frac{i}{2}\langle \bt,\theta QN\bu\rangle}\xi(\bu-N^{-1}Q^{-1}\bt)\\
%&=e^{-\frac{i}{2}\langle \bt, QN^{-1}\tilde{\theta}\bu\rangle}\xi(\bu-N^{-1}Q^{-1}\bt)\\
%&=e^{-\frac{i}{2}\langle N^{-1}Q^{-1}\bt,\tilde{\theta} \bu\rangle}\xi(\bu-N^{-1}Q^{-1}\bt)=(U^{\tilde{\theta}}(N^{-1}Q^{-1}\bt)\xi)(\bu).
%%\end{align*}
% Thus, 
% $$W^*U^\theta(\bt)W=U^{\tilde{\theta}}(N^{-1}Q^{-1}\bt),\quad \bt\in\bbR^d.$$
% Therefore, the claim follows.
 
 \end{proof}
 
 As an immediate corollary of the preceding proposition we obtain the following classical result (see e.g.  \cite[Proposition 5.2]{Rieffel} and also \cite[Proposition 2.13]{Gayral_Iochum_Varilly}).
 
\begin{corollary}\label{nc_plane_reduction_to_even}Let $S=\begin{pmatrix} 0&-1\\1&0\end{pmatrix}.$
We have that 
$$L_\infty(\bbR^d_\theta)\cong \underbrace{L_\infty(\bbR^2_S)\bar{\otimes}\cdots\bar{\otimes} L_\infty(\bbR^2_S)}_{\ell\,  {\text  times}}\bar{\otimes}\underbrace{ L_\infty(\bbR)\bar{\otimes}\dots\bar{\otimes}L_\infty(\bbR)}_{k \, \text{times}},$$ where $k=\dim\ker\theta$, $d=k+2\ell$.
\end{corollary}
\begin{proof} Let $\tilde{\theta}$  and $\{U^{\tilde{\theta}}(\bt)\}_{\bt\in\bbR^d}$ be the same as in Proposition \ref{nc_plane_to_block_diagonal}.

For $i=1,\dots, k$, $j=1,\dots, \ell$ and $t\in\bbR$ we set 
$$ u_j(t):=U^{\tilde{\theta}}(t\be_{2j}),\quad v_j(t)=U^{\tilde{\theta}}(t\be_{2j-1}), \quad w_i(t)=U^{\tilde{\theta}}(t\be_{2l+i}).$$ Then every pair $\{u_j(t)v_j(t)\}_{t\in\bbR}$, $j=1,\dots,\ell,$ generates a copy of $L_\infty(\bbR^2_S)$ and $\{w_i(t)\}_{t\in\bbR}$ generates $L_\infty(\bbR)$.

%, every $\cM_j, j=1,\dots \ell$, is generated by  two semigroups $\{u_j(t),v_j(t)\}_{t\in \bbR}$ satisfying $$u_j(t)v_j(s) = \exp(its)v_j(s)u_j(t)$$ and every $\cN_i, i=1,\dots,k$ is generated by a single unitary group.
%Since $\langle \be_{2j}, \tilde{\theta}\be_{n}\rangle=0$ for all $n\neq 2j-1$ and 
%$$\langle \be_{2j}, \tilde{\theta}\be_{2j-1}\rangle=\langle \be_{2j},\be_{2j}\rangle=1,$$
%commutation relation \eqref{nc_plane_comm_relation} implies that 
%$u_i, v_j$ commute with $u_n,v_j$ for $n\neq j$ and with $w_i$ for all $i=1\dots, k$, while
% $u_j(t)v_j(s)= \exp(its)v_j(s)u_j(t)$ for $j=1,\dots \ell$. 
% 
%For  every $j=1,\dots, \ell$ we set $\cM_j=\overline{\mathrm{span}(u_j(t), v_j(t),t\in\bbR)}^{wo}$, $j=1,\dots, \ell$, and for every  $i=1,\dots, k$ we set  $\cN_i=\overline{\mathrm{span}(w_i(t),t\in\bbR)}^{wo}$. 

\end{proof}

Corollary \ref{nc_plane_reduction_to_even} shows that we can always split $L^{\infty}(\bbR^d_\theta)$
into the tensor product $L_\infty(\bbR^2_S)\bar{\otimes}\cdots\bar{\otimes} L_\infty(\bbR^2_S)\bar{\otimes} L^{\infty}(\bbR^k)$, where $k$ is the dimension of the kernel of $\theta$. Thus truly noncommutative case appears only when $\theta$ is non-degenerate, that is $\det\theta\neq 0$. Therefore, everywhere below in this section we assume in addition that $\theta$ is non-degenerate and block-diagonal $d\times d$ matrix, with every block being equal to the matrix $S =\begin{pmatrix} 0&-1\\1&0\end{pmatrix}.$ 
In particular, this implies that $d$ is even.

Next, we show that in fact the algebra $L_{\infty}(\mathbb{R}_{\theta}^d)$ is $*$-isomorphic to the algebra $\cB(L_2(\mathbb{R}^{\frac{d}{2}}))$ (see \cite[Theorem 2]{GBV_JMP} and \cite[Proposition 2.13]{Gayral_Iochum_Varilly}).

  Let $M_{x_k}$ and $\partial_k$, $k=1,\dots,n$, \, $n\in\bbN$ be the operators on $L_2(\bbR^{n})$ given by $(M_{x_k} \xi)(\bu) = u_k\xi(\bu)$ and $(\partial_k\xi)(\bu) = -i\frac{\partial}{\partial u_k}\xi(\bu)$, $\bu\in\bbR^{d}$. 
  It is well-known (see e.g. \cite{RS_book_II}) that $(M_{x_1}, \dots, M_{x_d})$ and $(\partial_1,\dots, \partial_d)$ satisfy Weyl canonical commutation relations:
  \begin{align}\label{CCR}
  e^{isM_{x_k}}e^{itM_{x_l}}&=e^{itM_{x_l}}e^{isM_{x_k}},\nonumber\\
  e^{is\partial_k}e^{it\partial_l}&=e^{it\partial_l}e^{is\partial_k},\\
  e^{isM_{x_k}}e^{it\partial_l}&=e^{-its\delta_{kl}}e^{it\partial_l}e^{isM_{x_k}}.\nonumber
  \end{align}
  
  Introduce the quadratic form on $\bbR^d$ by setting
  \begin{equation}\label{nc_plane_quad_form}
  h(\bs)=\sum_j s_{2j}s_{2j-1}, \quad \bs\in \bbR^d.
  \end{equation}
Denoting $\bs'=(s_2,s_4,\dots, s_d),$ and $\bs''=(s_1,s_3,\dots, s_{d-1})$ for $\bs\in \bbR^d$, We have that
 \begin{equation}
 e^{i(\bs'\cdot M_x+\bs''\cdot\partial)}=e^{\frac{ih(\bs)}2}\prod_{j=1}^{d/2}e^{is_{2j}M_{x_{2j-1}}}\cdot \prod_{j=1}^{d/2}e^{is_{2j-1}\partial_{2j-1}}.
\end{equation}

\begin{theorem}\label{nc_plane_is_BL_2}
There exists a spatial $*$-isomorphisms $r$ of the algebras $L^\infty(\bbR^d_\theta)$ and $\cB(L_2(\bbR^{d/2}))\otimes 1_{L_2(\bbR^{d/2})}$ such that 
\begin{align*}
r:U(\bs)\mapsto e^{i(\bs'\cdot M_x+\bs''\cdot\partial)}, \quad \bs\in\bbR^d.
\end{align*}   
 \end{theorem}

\begin{proof} Recall that we assume that the matrix $\theta$ is block-diagonal (see Proposition \ref{nc_plane_to_block_diagonal}) of the form \eqref{nc_plane_theta_block-diagonal}.

Introduce the unitary operator $V\in\cB(L_2(\mathbb{R}^d))$ by setting $$(V\xi)(\bu)=\frac{1}{(2\pi)^{d/2}}\int_{\bbR^d}e^{i(-\frac{h(\bu)}2-\langle \bu,\theta\bx\rangle+h(\bx))}\xi(\bx)d\bx.$$
Standard computations show that 
  $$U(\bs)=Ve^{i(\bs'\cdot M_{\bx}+\bs''\cdot \partial)}V^*,\quad \bs\in\bbR^d.$$

\end{proof}

\begin{remark}
 Let $i_0:\mathcal{B}(L_2(\mathbb{R}^{\frac{d}{2}}))\to \mathcal{B}(L_2(\mathbb{R}^d))$ be defined by the formula
$$i_0(T)=T\otimes 1_{L_2(\mathbb{R}^{\frac{d}{2}})}.$$ By Theorem \ref{nc_plane_is_BL_2} we have that any $x\in L_\infty(\bbR_\theta^d)$ is in the image of $i_0$. Define the isomorphism $r_0:L_{\infty}(\mathbb{R}_{\theta}^d)\to\cB(L_2(\mathbb{R}^{\frac{d}{2}}))$ by setting
$$r_0(x)=i_0^{-1}(r(x)),\quad x\in L_{\infty}(\mathbb{R}_{\theta}^d).$$
Clearly, $r_0$ is a $*$-isomorphism of the algebras $L_{\infty}(\mathbb{R}_{\theta}^d)$ and $\mathcal{B}(L_2(\mathbb{R}^{d/2}))$.
\end{remark}

Having established the $*$-isomorphism $L_\infty(\bbR_\theta^d)\cong \cB(L_2(\bbR^{d/2}))$ we can now equip $L_\infty(\bbR_\theta^d)$ with a faithful normal semifinite trace $\tau_\theta$. Let $\tr$ be the standard trace on  $\cB(L_2(\bbR^{d/2}))$. We set 
\begin{equation}\label{nc_plane_trace}
\tau_\theta(x)=(\tr\otimes 1)(r(x))=(\tr\circ r_0)(x).
\end{equation}

In particular, we can define symmetric function spaces on  $L_\infty(\bbR_\theta^d)$, which are nothing but symmetric ideals of compact operators on $L_2(\bbR^{d/2})$ (see \cite{GK}). 

\begin{definition}Let $E$ be a symmetric sequence space. The symmetric ideal in  $L_\infty(\bbR_\theta^d)$ with respect to the trace $\tau_\theta$ is denoted by $E(\bbR_\theta^d,\tau_\theta)$. In particular, the Schatten ideals associated with $L_\infty(\bbR_\theta^d)$ are denoted by $L_p(\bbR_\theta^d,\tau_\theta)$
\end{definition}

\begin{remark}Since the algebra $L_{\infty}(\mathbb{R}_{\theta}^d)$ is isomorphic to $\mathcal{B}(L_2(\mathbb{R}^{d/2}))$ we have that $L_p(\bbR_\theta^d)\subset L_q(\bbR_\theta^d)$ for $p<q$.
\end{remark}

\begin{lemma}\label{nc_plane_L_2_description}
An operator $x\in L_\infty(\bbR_\theta^d)$ is in $L_2(\bbR_\theta^d,\tau_\theta)$ if and only if
$$x=\frac{1}{(2\pi)^{d/4}}\int_{\mathbb{R}^d}f(\bs)U(\bs)d\bs$$
for some unique $f\in L_2(\bbR^d)$ with $\|x\|_2=\|f\|_2.$ Here the integral is the Bochner integral.
\end{lemma}

We can also define the algebra of continuous functions in $L_\infty(\bbR_\theta^d)$.
\begin{definition}The $C^*$-algebra 
$$C(\bbR^d_\theta)=\{x\in L_\infty(\bbR^d_\theta): \, x \text{ is } \tau_\theta-\text{compact}\},$$
is said to be the algebra of continuous function (vanishing at infinity). This algebra is $*$-isomorphic to the algebra of compact operators on $L_2(\bbR^{d/2})$.
\end{definition}

\subsection{Calculus on $\bbR^d_\theta$}

In this section we define algebras of smooth functions on noncommutative Euclidean  space in $L_\infty(\bbR_\theta^d)$.

Let $\cD_k$,\, $1\leq k\leq d$ be multiplication operators 
$$(\cD_k\xi)(\bt)=t_k\xi(\bt), \quad \xi\in \dom(\cD_k)$$ defined on the domain $\dom(\cD_k)=\{\xi\in L_2(\bbR^d): \xi\in L_2(\bbR^d, |t|^2d\bt)\}.$ Note that the space $S(\bbR^d)$ is the core of every operator $\cD_k$, $k=1,\dots, d$.

Fix $\bs\in \bbR^d$ and let $\xi\in S(\bbR^d)$. We have 
\begin{align*}
([\cD_k, U(\bs)]\xi)(\bt)&=(\cD_kU(\bs)\xi)(t)-(U(\bs)\cD_k \xi)(t)\\
&=t_ke^{-\frac{i}2\langle \bs,\theta\bt\rangle }\xi(\bt-\bs)-e^{-\frac{i}2\langle \bs,\theta\bt\rangle }(t_k-s_k)\xi(\bt-\bs)\\
&=s_k(U(\bs)\xi)(\bt).
\end{align*}
Thus, for every $k=1,\dots, d$ and $s\in \bbR^d$ the operator $[\cD_k, U(\bs)]$ extend to a bounded operator on $L_2(\bbR^d)$ and 
\begin{equation}\label{nc_plane_com_D_k_U(s)}
[\cD_k, U(\bs)]=s_kU(\bs).
\end{equation}

In the next lemma we extend equation \eqref{nc_plane_com_D_k_U(s)} to commutators of $U(\bs)$ with the unitary group of $\cD_k$. 
\begin{lemma}\label{nc_plane_com_exp_D_k_U(s)}
For any $\bs\in\bbR^d$ and $k=1,\dots, d$ we have that  
$$e^{it\cD_k}U(\bs)e^{-it\cD_k}=e^{its_k}U(\bs)\in L_\infty(\bbR^d_\theta),\quad t>0.$$
\end{lemma}
\begin{proof}
Let $\xi\in L_2(\bbR^d)$. We have 
\begin{align*}
(e^{it\cD_k}U(\bs)\xi)(\bu)&=e^{-{i/2}\langle \bs,\theta \bu\rangle}e^{itu_k}\xi(\bu-\bs)\\
&=e^{its_k} \cdot e^{-{i/2}\langle \bs,\theta \bu\rangle}e^{it(u_k-s_k)}\xi(\bu-\bs)=e^{its_k}(U(\bs)e^{it\cD_k}\xi)(\bu).
\end{align*}

%We have 
%$$[e^{it\cD_k},U(\bs)]
%=\sum_{n\geq 1} \frac{(it)^n}{n!} [\cD_k^n,U(\bs)], \quad t>0.$$
%
%We firstly examine the commutators $[\cD_k^n,U(\bs)]$ for every fixed $n\geq 1$.
%We prove by induction that  
%\begin{equation}\label{nc_plane_com_powers_D_k_U(s)}
%[\cD_k^n,U(\bs)]=\sum_{j=0}^{n-1} C_n^j s_k^{n-j} U(\bs)\cD_k^j.
%\end{equation}
%For $n=1$ the assertion is already proved in \eqref{nc_plane_com_D_k_U(s)}.
%Now, assume that \eqref{nc_plane_com_powers_D_k_U(s)} holds for $n-1$. 
%We have 
%\begin{align*}
%[\cD_k^n,U(\bs)]&=\cD_k[\cD_k^{n-1},U(\bs)]+[\cD_k,U(\bs)]\cD_k^{n-1}\\
%&=\cD_k \sum_{j=0}^{n-2} C_{n-1}^j s_k^{n-1-j} U(\bs)\cD_k^j+s_kU(\bs)\cD_k^{n-1}\\
%&=\sum_{j=0}^{n-2} C_{n-1}^j s_k^{n-1-j}[\cD_k, U(\bs)]\cD_k^{j}+\sum_{j=0}^{n-2} C_{n-1}^j s_k^{n-1-j} U(\bs)\cD_k^{j+1}+s_kU(\bs)\cD_k^{n-1}\\
%&=\sum_{j=0}^{n-2} C_{n-1}^j s_k^{n-j}U(\bs)\cD_k^{j}+\sum_{j=1}^{n-1} C_{n-1}^{j-1} s_k^{n-j} U(\bs)\cD_k^{j}+s_kU(\bs)\cD_k^{n-1}.
%\end{align*}
%Combining similar term and using the Pascal's recurrence rule we obtain that
%$$[\cD_k^n,U(\bs)]=\sum_{j=0}^{n-1} C_n^j s_k^{n-j} U(\bs)\cD_k^j,$$
%as required. 
%
%Hence, 
%\begin{align*}
%[e^{it\cD_k},U(\bs)]
%&=\sum_{n\geq 1} \frac{(it)^n}{n!} \sum_{j=0}^{n-1} C_n^j s_k^{n-j}U(\bs)\cD_k^j
%\\
%&=U(\bs)\sum_{j\geq 0}\cD_k^j\sum_{n>j} \frac{(it)^ns_k^{n-j}}{(n-j)!j!}=U(\bs)\sum_{j\geq 0}\frac{(it)^j\cD_k^j}{j!} \sum_{n\geq 1} \frac{(its_k)^n}{(n)!}\\
%&=(e^{its_k}-1)U(\bs)e^{it\cD_k},
%\end{align*}
%which implies that 
%$$e^{it\cD_k}U(\bs)e^{-it\cD_k}=e^{its_k}U(\bs).$$
 \end{proof}

As a corollary of equation \eqref{nc_plane_com_D_k_U(s)} we obtain the following
\begin{proposition}\label{nc_plane_com_D_k_x}
 Let $k=1,\dots, d$ and  $x\in L_\infty(\bbR^d_\theta)$. If $[\cD_k,x]$ extends to a bounded operator on $L_2(\bbR^d)$, then $[\cD_k,x]\in L_\infty(\bbR^d_\theta)$.
\end{proposition}
\begin{proof}

Since $[\cD_k,x]\in \cB(L_2(\bbR^d))$, \cite[Proposition 3.2.55]{Brat_Rob}(see also \cite[Proposition 2.2]{AdePS_JFA_2005}) implies that 
$$i[\cD_k,x]=\lim_{t\to 0}\frac{ e^{it\cD_k}xe^{-it\cD_k}-x}{t},$$
with respect to the strong operator topology. 
Therefore, it is sufficient to show that $e^{it\cD_k}xe^{-it\cD_k}\in  L_\infty(\bbR^d_\theta)$ for any $t>0$. 

Let $\{x_i\}\subset \mathrm{span}\{U(\bs),\bs\in\bbR^d\}$ be such that $x_i\to x$ in the weak operator topology. By Lemma \eqref{nc_plane_com_exp_D_k_U(s)} we have that $e^{it\cD_k}x_ie^{-it\cD_k}\in  L_\infty(\bbR^d_\theta)$, and therefore $$e^{it\cD_k}xe^{-it\cD_k}=\lim_{i}e^{it\cD_k}x_ie^{-it\cD_k}\in L_\infty(\bbR_\theta^d),$$
with respect to the weak operator topology.
\end{proof}

The preceding proposition allows us to introduce mixed partial derivative $\partial^\alpha x$ of $x\in L_\infty(\bbR^d_\theta)$.

\begin{definition}Let $\alpha$ be a multiindex and let $x\in L_\infty(\bbR^d_\theta)$. If every repeated commutator $[\cD_{\alpha_j},[\cD_{\alpha_j+1},\dots,[\cD_{\alpha_n},x]]]$, $j=1,\dots n$ extends to a bounded operator on $L_2(\bbR^d)$, then the mixed partial derivative  $\partial^\alpha x$ of $x$ is defined as 
\begin{equation}
 \partial^{\alpha} x=[\cD_{\alpha_1},[\cD_{\alpha_2},\dots,[\cD_{\alpha_n},x]]].
\end{equation}
As usual, $\partial^0 x:=x$. 
By Proposition \ref{nc_plane_com_D_k_x}, $\partial^{\alpha} x\in L_\infty(\bbR_\theta^d)$. 
\end{definition}

Therefore, we can introduce the Sobolev space $W^{m,p}(\bbR_\theta^d)$ associated with noncommutative Euclidean  space. Note that our definition of Sobolev spaces is similar to noncommutative Sobolev spaces studied by Kissin and Shulman \cite{Kissin_Shulman}, which are called differential Schatten algebras. In our case, we take commutators with a family of self-adjoint operators $(\cD_{1},\dots,\cD_d)$, while for differential Schatten algebras \cite{Kissin_Shulman} considers commutators with a single symmetric (not necessarily self-adjoint) operator.

\begin{definition}
For a positive integer $m$ and $p \geq 1$, the space $W^{m,p}(\bbR^d_\theta)$
is the space of $x \in L^p(\bbR^d_\theta)$ such that every partial derivative
of $x$ up to order $m$ is also in $L^p(\bbR^d_\theta)$. This space
is equipped with the norm,
\begin{equation*}
    \|x\|_{W^{m,p}} = \sum_{|\alpha|\leq m} \|\partial^\alpha x\|_{p}
\end{equation*}
where the sum is taken over the set of multi-indices $\alpha$.
\end{definition}

The following proposition summarises some properties of Sobolev spaces. 
\begin{proposition}Let $m\in\bbN$. We have 
\begin{enumerate}
\item For any $m\in\bbN$ and $p\geq 1$ the Sobolev space $W^{m,p}(\bbR^d_\theta)$ is an algebra.
\item Let $p,q\geq 1$ with $\frac1p+\frac1q=1$ and $y\in W^{m,p}(\bbR_\theta^d), z\in W^{m,q}(\bbR_d^\theta)$. Then $yz\in W^{m,1}(\bbR_\theta^d).$
\item For $p<q$ we have $W^{m,p}(\bbR^d_\theta)\subset W^{m,q}(\bbR^d_\theta)$.

\item An operator $x\in L_\infty(\bbR_\theta^d)$ is in the Sobolev class $W^{m,2}(\bbR_\theta^d)$ if and only if
$$x=\frac{1}{(2\pi)^{d/4}}\int_{\bbR^d}f(\bs)U(\bs)d\bs $$
for some (unique) $f\in L_2(\bbR^d)$ such that $\bs\mapsto\prod_{n_k}s_k^{n_k}f(\bs)$, $\sum_k n_k\leq m,$ is square integrable. 
\item For every positive integer $m$ and $p \geq 1$ the Sobolev space $W^{m,p}(\bbR_\theta^d)$ is norm dense in $C(\bbR_\theta^d)$ (and hence, weakly dense in $L_\infty(\bbR^d_\theta)$).
\end{enumerate}

\end{proposition}

\begin{proof}
The first and the second assertions immediately follow from the Leibniz rule and H\"older inequality, while the third follows from the fact that $L_p(\bbR_\theta^d)\subset L_q(\bbR_\theta^d)$ for $p<q$. 
Part (iv) follows from Lemma \ref{nc_plane_L_2_description} and \eqref{nc_plane_com_D_k_U(s)}. Part (v) is an easy consequence of part (ii)-(iv). 

%(iv). By the definition of $W^{m,2}(\bbR_\theta^d)$, we have that $x\in L_2(\bbR_\theta^d)$, and therefore by Lemma \ref{nc_plane_L_2_description} there exists $f\in L_2(\bbR^d)$ such that 
%$x=\frac{1}{(2\pi)^{d/4}}\int_{\bbR^d}f(\bs)U(\bs)d\bs$. 
%We have 
%$$[D_k, x]=\frac{1}{(2\pi)^{d/4}}\int_{\bbR^d}f(\bs)[D_k,U(\bs)]d\bs,$$
%and therefore by \eqref{nc_plane_com_D_k_U(s)} we obtain that 
%$$[D_k, x]=\frac{1}{(2\pi)^{d/4}}\int_{\bbR^d}f(\bs)s_kU(\bs)d\bs.$$
%Using again Lemma \ref{nc_plane_L_2_description} we infer that $[D_k, x]\in L_2(\bbR_\theta^d)$ if and only if the function  $\bs\mapsto f(\bs)s_k$ is square-integrable.  Repeating the argument we arrive at the assertion.

%(v). By part (iii) it is sufficient to show that $W^{m,1}(\bbR_\theta^d)$ is dense in $C(\bbR_\theta^d)$. 
%
%Let $x\in L_1(\bbR_\theta^d)$. Then $x=x_1x_2$ for some $x_1,x_2\in L_2(\bbR_\theta^d)$. Part (iv) implies that there exist $x_{1,n}, x_{2,n}\in W^{m,2}$, such that $\|x_i-x_{i,n}\|_2\to 0$ as $n\to\infty$ for $i=1,2$. Therefore, $x_{1,n}x_{2,n}\to x_1x_2$ in $L_1(\bbR^d_\theta)$ and by part (ii) $x_{1,n}x_{2,n}\in W^{m,1}(\bbR_\theta^d)$. Hence $W^{m,1}(\bbR_\theta^d)$ is dense in $L_1(\bbR^d_\theta)$, and therefore in $C(\bbR^d_\theta)$.

\end{proof}

Next we introduce the Laplace and gradient operators associated with $L_\infty(\bbR^d_\theta)$.
\begin{definition}\label{nc_plane_gradient}
  The Laplace operator $\Delta_\theta$ associated with  $L_\infty(\bbR_\theta^d)$ is defined on the domain $\dom(-\Delta)=L_2(\bbR,|\bt|^4d\bt)$ by 
$$(-\Delta_\theta \xi)(\bt)=(\sum_{j=1}^d t_j^2)\xi(t).$$

The gradient $\nabla_\theta$ associated with  $L_\infty(\bbR_\theta^d)$ is the operator acting in $L_2(\bbR^d)^d$ 
by 
$$\nabla_\theta=(-i\cD_1,\dots, -i\cD_d)=(-iM_{t_1},\dots, -iM_{t_d})$$ with the domain 
$\dom(\nabla_\theta)=L_2(\bbR^d, t_1^2d\bt)\times \dots \times L_2(\bbR^d, t_d^2d\bt).$
\end{definition}

\begin{definition}
  Let $N=2^{\lfloor d/2\rfloor}$ and let $\{\gamma_j\}_{j=1}^d$ be $d$-dimensional $\gamma$-matrices satisfying $\gamma_i\gamma_k+\gamma_k\gamma_j=2\delta_{j,k}$.  The Dirac operator  $\cD$ associated with  $L_\infty(\bbR^d_\theta)$ is the operator on $\bbC^N\otimes L^2(\bbR^d)$ defined by
    \begin{equation*}
        \cD:= \sum_{j=1}^d \gamma_j\otimes \cD_j=\sum_{j=1}^d \gamma_j\otimes M_{t_j}.
    \end{equation*}
    
The associated spectral triple is $(1\otimes W^{m,1}(\bbR_\theta^d), \bbC^N\otimes L_2(\bbR^d), \cD)$.
\end{definition}
\begin{remark} It is clear the operators $\Delta_\theta$, $\nabla_\theta$ do not depend on the matrix $\theta$. However, we prefer to use notation with $\theta$ to distinguish  $\Delta_\theta$, $\nabla_\theta$ from the classical Laplacian $\Delta$ and gradient $\nabla$.
\end{remark} 

\section{Cwikel estimates for the noncommutative Euclidean  space}\label{sec_nc_plane_Cwikel}

In this section we focus our attention on operators of the form $xg(-i\nabla_\theta)$. 
Everywhere below we assume that $\theta$ is non-degenerate, real antisymmetric matrix.
We start with Cwikel estimates in Hilbert-Schmidt class (see \cite[Lemma 4.3]{Gayral_Iochum_Varilly}).

\begin{lemma}\label{nc_plane_l2_cwikel} If $x\in L_2(\mathbb{R}^d_{\theta},\tau_{\theta})$ and if $g\in L_2(\mathbb{R}^d),$ then $xg(-i\nabla_\theta)\in \cL_2(L_2(\bbR^d))$ and 
$$\|x\, g(-i\nabla_\theta)\|_2=\frac{1}{(2\pi)^{d/4}}\|x\|_2\|g\|_2.$$
\end{lemma}
\begin{proof} By Lemma \ref{nc_plane_L_2_description} there exists $f\in L_2(\mathbb{R}^d)$ such that 
$$x=\frac{1}{(2\pi)^{d/4}}\int_{\mathbb{R}^d}f(\bs)U(\bs)d\bs.$$
It follows from the definition of $\nabla_\theta$, that the operator $g(-i\nabla_\theta)$ is multiplication operator $M_g$ on $L_2(\bbR^d)$. 
Therefore,
% we have 
%\begin{align*}
%(x\, g(-i\nabla_\theta)\xi)(t)&=\frac{1}{(2\pi)^{d/4}}\int_{\bbR^d}f(\bs)(U(\bs))(g\xi)(\bt)d\bs\\
%&=\frac{1}{(2\pi)^{d/4}}\int_{\mathbb{R}^d}f(\bs)e^{-\frac{i}{2}\langle \bs,\theta \bt\rangle}g(\bt-\bs)\xi(\bt-\bs)ds\\
%&=\frac{1}{(2\pi)^{d/4}}\int_{\mathbb{R}^d}f(\bt-\bs)e^{\frac{i}{2}\langle \bs,\theta \bt\rangle}g(\bs)\xi(\bs)d\bs.
%\end{align*}
%Thus, 
$x g(-i\nabla_\theta)$ is an integral operator on $L_2(\bbR^d)$ with the kernel 
$$(\bt,\bs)\mapsto f(\bt-\bs)g(\bs)e^{\frac{i}{2}\langle \bs,\theta \bt\rangle}.$$
%We have 
%$$\iint |f(\bt-\bs)|^2|g(\bs)|^2d\bs d\bt=\|f\|_2^2\|g\|_2^2<\infty.$$
Hence,
$$\|x g(-i\nabla_\theta)\|_2=\frac{1}{(2\pi)^{d/4}}\|f\|_2\|g\|_2=\frac{1}{(2\pi)^{d/4}}\|x\|_2\|g\|_2<\infty,$$
where the last equality follows from Lemma \ref{nc_plane_L_2_description}.
\end{proof}

As an immediate corollary of Lemma \ref{nc_plane_l2_cwikel} and Theorem \ref{abs_Cwikel_main_thm} we have the following result, which yields, in particular, Cwikel estimates for noncommutative Euclidean  space in Schatten and weak Schatten ideals for $p>2$.

\begin{theorem}\label{nc_plane_Cwikel_p>2}
 Let $(E(0,\infty),\|\cdot\|_E)$ be an interpolation space for $(L_2,L_\infty)$. If $x\otimes g \in E(L_\infty(\bbR^d_\theta)\otimes L_\infty(\bbR^d))$, then $xg(-i\nabla_\theta)\in E(L_2(\bbR^d))$ and 
 $$\|xg(-i\nabla_\theta)\|_{E(L_2(\bbR^d))}\leq C_E \|x\otimes g\|_{E(L_\infty(\bbR^d_\theta)\otimes L_\infty(\bbR^d))}.$$
\end{theorem}
\begin{proof}
To employ Theorem \ref{abs_Cwikel_main_thm} we consider the algebras $\cA_1=L_\infty(\bbR_\theta^d)$ and $\cA_2=L_\infty(\bbR^d)$. Let us represent the algebra $L_\infty(\bbR_\theta^d)$ on $L_2(\bbR^d)$ by the identity representation $\pi_1(x)=x, x\in L_\infty(\bbR_\theta^d)$  and the algebra $L_\infty(\bbR^d)$ by the representation $\pi_2(g)=g(-i\nabla_\theta)$. By Lemma \ref{nc_plane_l2_cwikel} the Hypothesis \ref{abs_Cwikel_hyp} is satisfied in this case. Therefore, the result follows from Theorem \ref{abs_Cwikel_main_thm}.
\end{proof}

As a corollary of Theorem \ref{nc_plane_Cwikel_p>2} we can obtain an analogue of Cwikel-Lieb-Rosenblum inequality for noncommutative Euclidean space for $d\geq 4$. We denote by $N(-\Delta_\theta+x)$ the number of negative eigenvalues of the operator $-\Delta_\theta+x$, where $x=x^*\in L_{d/2}(\bbR_\theta^d)$. Repeating the arguments from Cwikel's proof of Cwikel-Lieb-Rosenblum inequality \cite{Cwikel}, we obtain the following
\begin{corollary}\label{nc_plane_CLR}
Let $d\geq 4$, $x=x^*\in L_{d/2}(\bbR_\theta^d)$ and let $x_-$ be the negative part of $x$. Then 
$$N(-\Delta_\theta+x)\leq C_d\cdot \tau_\theta(|x_-|^{d/2}),$$
with the constant $C_d$ being equal to the constant in \cite{Cwikel}.
\end{corollary}

In the rest of the section we establish Cwikel estimates in Schatten and weak Schatten ideals for for $1\leq p\leq 2$.
We begin with an auxiliary lemma.

\begin{lemma}\label{nc_plane_Cwikel_L1} Let $1\leq p\leq 2$, $k\geq 0$ and let $x\in W^{k,p}(\mathbb{R}^d_{\theta}).$ Then
$$(1-\Delta_\theta)^{\frac{k-1}{2}-\frac{d}{4}} x(1-\Delta_\theta)^{-\frac{k+1}2-\frac{d}{4}}\in\cL_p(L_2(\bbR^d))$$ and 
$$\|(1-\Delta_\theta)^{\frac{k-1}{2}-\frac{d}{4}} x(1-\Delta_\theta)^{-\frac{k+1}2-\frac{d}{4}}\|_p\leq 2^k\|x\|_{W^{k,p}}.$$
\end{lemma}
\begin{proof} Note that by the definitions of $\cD$ and $-\Delta_\theta$, we have that $\cD^2=-1\otimes \Delta_\theta$, and therefore  the assertion is equivalent to the fact that
$(\cD-i)^{k-1-\frac{d}{2}} (1\otimes x)(\cD-i)^{-k-1-\frac{d}{2}}\in\cL_p(\bbC^N\otimes L_2(\bbR^d))$ and 
\begin{equation}\label{nc_plane_Cwikel_L_1_equivalent_form}
\|(\cD-i)^{k-1-\frac{d}{2}} (1\otimes x)(\cD-i)^{-k-1-\frac{d}{2}}\|_{\cL_p(\bbC^N\otimes L_2(\bbR^d))}\leq 2^k\|x\|_{W^{k,p}}.
\end{equation}

We prove the assertion by induction on $k.$ For $k=0,$ 
we have
\begin{align*}
(1-\Delta_\theta)^{-\frac{1}{2}-\frac{d}{4}} x(1-\Delta_\theta)^{-\frac{1}2-\frac{d}{4}}
=(1-\Delta_\theta)^{-\frac{d}{4}-\frac12}|x|^{\frac12}\cdot\sgn(x)\cdot|x|^{\frac12}(1-\Delta_\theta)^{-\frac{d}{4}-\frac12}.
\end{align*}
Since $|x|^{1/2}\in L_{2p}(\bbR_\theta^d)$ and $(1+|\cdot|^2)^{-\frac{1}{2}-\frac{d}{4}}\in L_{2p}(\bbR^d)$, Lemma \ref{nc_plane_l2_cwikel} implies that 
$(1-\Delta_\theta)^{-\frac{d}{4}-\frac12}|x|^{\frac12}, |x|^{\frac12}(1-\Delta_\theta)^{-\frac{d}{4}-\frac12}\in \cL_{2p}(L_2(\bbR^d))$. Therefore, the H\"older inequality \eqref{Holder} implies that $(1-\Delta_\theta)^{-\frac{1}{2}-\frac{d}{4}}x(1-\Delta_\theta)^{-\frac{1}2-\frac{d}{4}}\in\cL_{p}(L_2(\bbR^d))$ and 
\begin{align*}
\|(1-\Delta_\theta)^{-\frac{1}{2}-\frac{d}{4}}&x(1-\Delta_\theta)^{-\frac{1}2-\frac{d}{4}}\|_{\cL_p(L_2(\bbR^d))}\\
&\leq\|(1-\Delta_\theta)^{-\frac{d}{4}-\frac12}|x|^{\frac12}\|_{\cL_{2p}(L_2(\bbR^d))}\||x|^{\frac12}(1-\Delta_\theta)^{-\frac{d}{4}-\frac12}\|_{\cL_{2p}(L_2(\bbR^d))}\\
&=c_d\||x|^{\frac12}\|_{2p}^2=c_d\|x\|_p.
\end{align*}
Thus, the assertion holds for $k=0$. 

Now suppose that the assertion holds for $k\geq 0$. We prove the equivalent assertion \eqref{nc_plane_Cwikel_L_1_equivalent_form}.
We have
\begin{align*}
(\cD-i)^{k-\frac{d}{2}}&(1\otimes x)(\cD-i)^{-k-\frac{d}{2}-2}\\
&=(\cD-i)^{k-\frac{d}{2}-1}(1\otimes x)(\cD-i)^{-k-\frac{d}{2}-1}\\
&\quad\quad+(\cD-i)^{k-\frac{d}{2}}[\frac1{\cD-i},1\otimes x](\cD-i)^{-k-\frac{d}{2}-1}\\
&=(\cD-i)^{k-\frac{d}{2}-1}(1\otimes x)(\cD-i)^{-k-\frac{d}{2}-1}\\
&\quad\quad-(\cD-i)^{k-\frac{d}{2}-1}[\cD,1\otimes x](\cD-i)^{-k-\frac{d}{2}-2}.
\end{align*}
We set 
\begin{align*}
T_1&=(\cD-i)^{k-\frac{d}{2}-1}(1\otimes x)(\cD-i)^{-k-\frac{d}{2}-1},\\
T_2&=(\cD-i)^{k-\frac{d}{2}-1}[\cD,1\otimes x](\cD-i)^{-k-\frac{d}{2}-2}.
\end{align*}
By inductive hypothesis, $T_1\in\mathcal{L}_{p}(L_2(\bbR^d))$ and
$$\|T_1\|_p\leq 2^k\|x\|_{W^{k,p}}\leq 2^k\|x\|_{W^{k+1,p}}.$$

For $T_2$ we have 
$$T_2=\sum_{l=1}^d(\cD-i)^{k-\frac{d}{2}-1}(\gamma_l\otimes [\cD_l, x])(\cD-i)^{-k-\frac{d}{2}-2}.$$
By triangle inequality, we have
\begin{align*}
\|T_2\|_p&\leq\sum_{l=1}^d\|(\cD-i)^{k-\frac{d}{2}-1}(\gamma_l\otimes [\cD_l, x])(\cD-i)^{-k-\frac{d}{2}-2}\|_p\\
&=\sum_{l=1}^d\|(\cD-i)^{k-\frac{d}{2}-1}(1\otimes [\cD_l, x])(\cD-i)^{-k-\frac{d}{2}-2}\|_p\\
&\leq\sum_{l=1}^d\|(\cD-i)^{k-\frac{d}{2}-1}(1\otimes [\cD_l, x])(\cD-i)^{-k-\frac{d}{2}-1}\|_p.
\end{align*}
Using induction hypothesis (as applied to the element $[\cD_l, x]$), we obtain
$$\|T_2\|_p\leq 2^k\sum_{l=1}^d\|[\cD_l, x]\|_{W^{k,p}}\leq 2^k\|x\|_{W^{k+1,p}}.$$

It follows that
$$\|(D-i)^{k-\frac{d}{2}}(1\otimes x)(D-i)^{-k-\frac{d}{2}-2}\|_p\leq 2^{k+1}\|x\|_{W^{k+1,p}}.$$
\end{proof}

The following lemma shows that in the noncommutative Euclidean  space we have a version of translation invariance. 

\begin{lemma}\label{nc_plane_translation_invariance}
For every $\bt\in\mathbb{R}^d,$ there exists a unitary element $V(\bt)$ on $L_2(\mathbb{R}^d)$ such that
$$[V(\bt), x]=0,\quad x\in L_{\infty}(\mathbb{R}^d_{\theta})$$
and 
$$V(\bt)e^{is\cD_k}V(\bt)^{*}=e^{is(\cD_k+t_k)},$$
in particular $\cD_k$ is unitary equivalent to $\cD_k+t_k$.
\end{lemma}
\begin{proof} 
Without loss of generality we assume $\theta$ is block-diagonal $d\times d$ matrix, with every block equal to the matrix $S =\begin{pmatrix} 0&-1\\1&0\end{pmatrix}$ (otherwise apply $W$ defined in \eqref{nc_plane_change_coordinaes}).

Since for $k\neq n$, $k,n=1,\dots d$ the operators $\partial_k$ and $M_{x_n}$ commute, we can consider strongly continuous unitary groups defined by
\begin{align*}
e^{itA_{2j-1}}&=e^{it\partial_{2j-1}}e^{\frac{it}{2}M_{x_{2j}}},\\
e^{itA_{2j}}&=e^{it\partial_{2j}}e^{-\frac{it}{2}M_{x_{2j-1}}},\quad j=1,\dots \frac{d}2,
\end{align*}
where $A_k, k=1,\dots, d$ denotes the generator of respective group. 
We note that by Weyl canonical commutation relations \eqref{CCR} we have 
 $$e^{itA_{2k}}e^{isA_{2k-1}}=e^{its}e^{isA_{2k-1}}e^{itA_{2k}}$$ and $[e^{itA_k}, e^{isA_n}]=0$ otherwise.

%
%
%For any $1\leq k\leq d$ we define the (unbounded) operator on $\dom(A_k)=S(\bbR^d)\subset L_2(\bbR^d)$ by setting
%$$A_k=\frac1i\partial_k-\frac12\sum_{l=1}^d\theta_{kl}M_{t_l},\quad,1\leq k\leq d.$$
%It follows from \cite[Theorem VIII.14 and its Corollary]{RS80}, that $A_k$ is essentially self-adjoint operator for every $1\leq k\leq d$. In what follows we preserve the notation $A_k$ for the self-adjoint extension of $A_k$.  

Since $\cD_k=M_{t_k}$ it follows that 
\begin{equation}\label{nc_plane_commut_A_k_D_l}
e^{itA_k}e^{i\cD_l}=e^{is\cD_l}e^{itA_k},\quad l\neq k\\
\end{equation}
and 
\begin{align}\label{nc_plane_commut_A_k_D_k}
e^{itA_k}e^{is\cD_k}&=e^{it\partial_k}e^{(-1)^{k-1}\frac{it}2 M_{x_{k+(-1)^{k-1}}}}e^{isM_{x_k}}\nonumber\\
&=e^{_its} e^{isM_{x_k}}e^{it\partial_k}e^{(-1)^{k-1}\frac{it}2 M_{x_{k+(-1)^{k-1}}}}\\
&=e^{its}e^{is\cD_k}e^{itA_k},\quad 1\leq k\leq n.\nonumber
\end{align}

Moreover, for any $\xi\in L_2(\bbR^d)$  we have 
\begin{align*}
([e^{itA_{2j}},&U(\bs)]\xi)(\bu)=(e^{itA_{2j}}U(\bs) \xi)(\bu)-(U(\bs)e^{itA_{2j}} \xi)(\bu)\\
%&=e^{-i\frac{t}{2}u_{2j-1}}e^{-\frac{i}2\langle \bs,\theta(\bu+t\be_{2j})\rangle}\xi(\bu-\bs+t\be_{2j})\\
%&\quad-e^{-i\frac{t}{2}(u_{2j-1}-s_{2j-1})}e^{-\frac{i}2\langle \bs,\theta\bu\rangle}\xi(\bu-\bs+t\be_{2j})\\
&=\Big(e^{-i\frac{t}{2}u_{2j-1}}e^{-\frac{it}2\langle \bs,\theta\be_{2j}\rangle}-e^{-i\frac{t}{2}(u_{2j-1}-s_{2j-1})}\Big)e^{-\frac{i}2\langle \bs,\theta\bu\rangle}\xi(\bu-\bs+t\be_{2j}).
%\\
%e^{-\frac{i}2\langle \bs,\theta \bt\rangle}\xi(\bt-\bs)-\frac12e^{-\frac{i}2\langle \bs,\theta \bt\rangle}\sum_{l=1}^d\theta_{kl}t_l\xi(\bt-\bs)\\
%&\quad-\frac1ie^{-\frac{i}2\langle \bs,\theta \bt\rangle}\xi_k(\bt-\bs)+\frac12e^{-\frac{i}2\langle \bs,\theta \bt\rangle}\sum_{l=1}^d\theta_{kl}(t_l-s_l)\xi(\bt-\bs)\\
%&=\frac1i \partial_ke^{-\frac{i}2\langle \bs,\theta \bt\rangle}\xi(\bt-\bs)-\frac1ie^{-\frac{i}2\langle \bs,\theta \bt\rangle}\xi_k(\bt-\bs)-\frac12e^{-\frac{i}2\langle \bs,\theta \bt\rangle}\sum_{l=1}^d\theta_{kl}s_l\xi(\bt-\bs).
%%\\
%%&=-\frac12 \sum_{l=1}^d s_l\theta_{kl}e^{-\frac{i}2\langle \bs,\theta \bt\rangle}\xi(\bt-\bs)+\frac1ie^{-\frac{i}2\langle \bs,\theta \bt\rangle}\xi_k(\bt-\bs)\\
%%&\quad-\frac1ie^{-\frac{i}2\langle \bs,\theta \bt\rangle}\xi_k(\bt-\bs)+\frac12e^{-\frac{i}2\langle \bs,\theta \bt\rangle}\sum_{l=1}^d\theta_{kl}(t_l-s_l)\xi(\bt-\bs).
\end{align*}
Since $e^{-\frac{it}2\langle \bs,\theta\be_{2j}\rangle}=e^{\frac{it}2s_{2j-1}},$ we conclude that 
$$[e^{itA_{2j}},U(\bs)]=0.$$
Repeating similar argument for $[e^{itA_{2j}-1},U(\bs)]$ one can obtain that 
\begin{equation}\label{nc_plane_commut_A_k_pi}
[e^{itA_{j}},U(\bs)]=0,\quad j=1,\dots d, \quad t\in\bbR, \quad\bs\in\mathbb{R}^d.
\end{equation}
 
For $\bt=(t_1,\dots, t_d)$ we set 
$$V(\bt)=\prod_{k=1}^d\exp(it_kA_k).$$
Clearly $V(\bt)$ is a unitary operator for any $\bt\in\bbR^d$. 

Equality \eqref{nc_plane_commut_A_k_pi} implies that 
$[V(\bt),U(\bs)]=0,$ for any $\bs,\bt\in\bbR^d$, and therefore $[V(\bt), x]=0,$ for any $x\in L_\infty(\bbR^d_\theta)$.

On the other hand, equations \eqref{nc_plane_commut_A_k_D_l} and \eqref{nc_plane_commut_A_k_D_k} imply that 
\begin{align*}
V(\bt)e^{is\cD_k}V(\bt)^{-1}&=\prod_{j=1}^d\exp(it_jA_j) e^{is\cD_k}\prod_{l=0}^{d-1}\exp(-it_lA_l)\\
&=e^{it_ks}e^{is\cD_k}=e^{is(\cD_k+t_k)}.
\end{align*}
\end{proof}

We are now ready to prove the main result of the present section. The estimates we obtain significantly extend the estimates obtained in \cite{Gayral_Iochum_Varilly} and \cite{CGRS_Memoirs}.

\begin{theorem}\label{nc_plane_Cwikel_weakL1}Let $1\leq p\leq 2$.
For every $x\in W^{d,p}(\bbR^d_\theta)$ and $g\in \ell_{p,\infty}(L_\infty)(\bbR^d)$ we have that $x g(-i\nabla_\theta)\in \cL_{p,\infty}(L_2(\bbR^d))$ and 
$$\|x g(-i\nabla_\theta)\|_{\cL_{p,\infty}(L_2(\bbR^d))}\leq \const \|x\|_{W^{d,p}}\|g\|_{\ell_{p,\infty}(L_\infty)}.$$
\end{theorem}
\begin{proof} For every $\bn\in\mathbb{Z}^d,$ we set
$$B_\bn= xg(-i\nabla_\theta)\chi_{\bn+K}(-i\nabla_\theta).$$

By Lemma \ref{nc_plane_l2_cwikel} we have
$$\sum_{\bn}\|B_\bn\|_2^2=\sum_\bn \|x\|_2^2\|g\chi_{K+\bn}\|_2^2\leq\|x\|_2^2\|g\|_{\ell_{p,\infty}(L_\infty)}<\infty.$$
Hence, since, $\{B_\bn\}$ are disjoint from the right, Lemma \ref{lem_strong_major} implies that 
\begin{equation}\label{nc_plane_via_direct_sum}
\mu^2(\bigoplus_{\bn\in\mathbb{Z}^d}B_\bn)\prec \mu^2(\sum_{\bn\in\mathbb{Z}^d}B_\bn)=\mu^2(x g(-i\nabla_\theta)).
\end{equation}

We claim that $\bigoplus_{\bn\in\mathbb{Z}^d}B_\bn\in \cL_{p,\infty}(L_2(\bbR^d))$. 

For every $\bn\in\mathbb{Z}^d,$ we set
$$A_\bn:=x\chi_{\bn+K}(-i\nabla_\theta).$$
We have 
\begin{align*}
\mu(B_\bn)&=\mu\big(A_\bn\cdot g(-i\nabla_\theta)\chi_{\bn+K}(-i\nabla_\theta)\big)\leq \big\|g(-i\nabla_\theta)\chi_{\bn+K}(-i\nabla_\theta)\big\|_\infty \mu(A_\bn)\\
&\leq \|g\chi_{K+n}\|_\infty\mu(A_\bn).
\end{align*}
Therefore,
$$\Big\|\bigoplus_{\bn\in\mathbb{Z}^d}B_\bn\Big\|_{p,\infty}\leq \Big\|\bigoplus_{\bn\in\mathbb{Z}^d}\|g\chi_{K+\bn}\|_\infty A_\bn\Big\|_{p,\infty}.$$

Let $V(\bn)$ be the unitary operator constructed in Lemma \ref{nc_plane_translation_invariance}. We have
$$V(\bn)A_\bn V(\bn)^*=xV(\bn)\chi_{\bn+K}(-i\nabla_\theta)V(\bn)^*=x\chi_{K}(-i\nabla_\theta)=A_0.$$
 Hence, by Lemma \ref{tensor_L_1_weak_L_1} we obtain  
\begin{align*}
\Big\|&\bigoplus_{\bn\in\mathbb{Z}^d}B_\bn\Big\|_{p,\infty}\leq \Big\|\bigoplus_{\bn\in\mathbb{Z}^d}\|g\chi_{K+\bn}\|_\infty V(\bn)A_\bn V(\bn)^*\Big\|_{p,\infty}\\
&=c_d\Big\|A_0\otimes\Big\{\|g\chi_{K+\bn}\|_\infty\Big\}_{\bn\in\mathbb{Z}^d}\Big\|_{p,\infty}\leq \|A_0\|_p\cdot\Big\|\Big\{|g\chi_{K+\bn}\|_\infty\Big\}_{\bn\in\mathbb{Z}^d}\Big\|_{p,\infty}\\
&=\|A_0\|_p\|g\|_{\ell_{p,\infty}(L_\infty)}.
\end{align*}
We have 
\begin{align*}
\mu(A_0)&\leq \mu(x (1-\Delta_\theta)^{-\frac{d}{2}-1})\Big\|(1-\Delta_\theta)^{\frac{d}{2}+1}\prod_{k=1}^d\chi_{(0,1)}(\cD_k)\Big\|_\infty\\
&=(1+d)^{\frac{d}{2}+1}\mu( x(1-\Delta_\theta)^{-\frac{d}{2}-1}),
\end{align*}
in particular, 
$$\|A_0\|_p\leq \const \|x(1-\Delta_\theta)^{-\frac{d}{2}-1}\|_p.$$
It follows from Lemma \ref{nc_plane_Cwikel_L1} (for $k=1+\frac{d}{2}$) that
$$\|A_0\|_{p,\infty}\leq c_d\|x\|_{W^{\frac{d}{2}+1,1}}\leq c_d\|x\|_{W^{d,p}}.$$
Thus, $\bigoplus_{\bn\in\mathbb{Z}^d}B_\bn\in\cL_{p,\infty}(L_2(\bbR^d))$ and 
$$\Big\|\bigoplus_{\bn\in\mathbb{Z}^d}B_\bn\Big\|_{p,\infty}\leq c_d\|x\|_{W^{d,p}}\|g\|_{\ell_{p,\infty}(L_\infty)}.$$

Hence, combining Proposition \ref{clas_Lp_weakLp} with \eqref{nc_plane_via_direct_sum} we infer that
 $$x g(-i\nabla_\theta)\in \cL_{p,\infty}(L_2(\bbR^d))$$ and 
$$\|x g(-i\nabla_\theta)\|_{\cL_{p,\infty}(L_2(\bbR^d))}\leq C_p\Big\|\bigoplus_{\bn\in\mathbb{Z}^d}B_\bn\Big\|_{p,\infty}\leq \const \|x\|_{W^{d,p}}\|g\|_{\ell_{p,\infty}(L_\infty)}.$$
\end{proof}

To conclude this section, we prove Cwikel estimates for the Schatten ideals $\cL_p(L_2(\bbR^d))$, for $1\leq p<2$.

\begin{theorem}Let $1\leq p<2$. For every $x\in W^{d,p}(\mathbb{R}^d_{\theta})$ and for every $g\in l_p(L_{\infty})(\bbR^d),$ we have that $xg(-i\nabla_\theta)\in\mathcal{L}_p(L_2(\mathbb{R}^d))$ and
$$\|xg(-i\nabla_\theta)\|_{\mathcal{L}_p(L_2(\mathbb{R}^d))}\leq{\rm const}\|x\|_{W^{d,p}(\mathbb{R}^d_{\theta})}\|g\|_{l_p(L_{\infty})}.$$
\end{theorem}
\begin{proof} Since $p\geq 2$, we have that $x\in \cL_2(\mathbb{R}^d_{\theta})$ and $g\in L_2(\bbR^d)$. Hence, by Lemma \ref{nc_plane_l2_cwikel}, we have that $xg(-i\nabla_\theta)\in \cL_2(L_2(\bbR^d))$.
 Let, as before, $K$ be the unit cube in $\bbR^d$. As in the proof of Theorem \ref{nc_plane_Cwikel_weakL1}, we have 
\begin{equation}\label{nc_plane_Cwikel_L1_decomp}
\mu^2\Big(\bigoplus_{\bn\in\bbZ^d}xg(-i\nabla_{\theta})\chi_{{\bf n}+K}(-i\nabla_\theta)\Big)\prec \mu^2\Big(\sum_{{\bf n}\in\mathbb{Z}^d}xg(-i\nabla_{\theta})\chi_{{\bf n}+K}(-i\nabla_\theta)\Big)=\mu^2(xg(-i\nabla_\theta)).
\end{equation}

By Lemma \ref{nc_plane_translation_invariance} the operators $x\chi_{{\bf n}+K}(-i\nabla_\theta)$ and $x\chi_K(-i\nabla_\theta)$ are unitarily equivalent, and therefore,
\begin{align*}
\|xg(-i\nabla_\theta)\chi_{{\bf n}+K}(-i\nabla_\theta)\|_{\mathcal{L}_p(L_2(\mathbb{R}^d))}\leq\|g\chi_{{\bf n}+K}\|_{\infty}\cdot\|x\chi_{{\bf n}+K}(-i\nabla_\theta)\|_{\mathcal{L}_p(L_2(\mathbb{R}^d))}\\
=\|g\chi_{{\bf n}+K}\|_{\infty}\cdot\|x\chi_{K}(-i\nabla_\theta)\|_{\mathcal{L}_p(L_2(\mathbb{R}^d))}.
\end{align*}
By Lemma \ref{nc_plane_Cwikel_L1} (for $k=1+\frac{d}2$) we have 
\begin{align*}
\|x\chi_K(-i\nabla_\theta)\|_{\mathcal{L}_p(L_2(\mathbb{R}^d))}&\leq \|(1-\Delta_\theta)^{\frac{d+1}{2}}\chi_K(-i\nabla_\theta)\|_\infty\|x(1-\Delta_\theta)^{-\frac{d+1}{2}}\|_{\mathcal{L}_p(L_2(\mathbb{R}^d))}\\
&\leq{\rm const}\|x\|_{W^{d,p}(\mathbb{R}^d_{\theta})}.
\end{align*}

Therefore, we obtain that 
\begin{align*}
\Big\|\bigoplus_{\bn\in\bbZ^d}&xg(-i\nabla_{\theta})\chi_{{\bf n}+K}(-i\nabla_\theta)\Big\|_p=
\sum_{{\bf n}\in\mathbb{Z}^d}\|xg(-i\nabla_\theta)\chi_{{\bf n}+K}(-i\nabla_\theta)\|^p_{\mathcal{L}_p(L_2(\mathbb{R}^d))}\\
&\leq \const \|x\|^p_{W^{d,p}(\mathbb{R}^d_{\theta})} \sum_{{\bf n}\in\mathbb{Z}^d}\|g\chi_{{\bf n}+K}\|^p_{\infty}=\const\|x\|^p_{W^{d,p}(\mathbb{R}^d_{\theta})}\|g\|^p_{\ell_p(L_\infty)(\bbR^d)}.
\end{align*}

Hence, Proposition \ref{clas_Lp_weakLp} (i) together with \eqref{nc_plane_Cwikel_L1_decomp} implies that $xg(-i\nabla_\theta)\in\mathcal{L}_p(L_2(\mathbb{R}^d))$ and
$$\|xg(-i\nabla_\theta)\|_{\mathcal{L}_p(L_2(\mathbb{R}^d))}\leq{\rm const}\|x\|_{W^{d,p}(\mathbb{R}^d_{\theta})}\|g\|_{l_p(L_{\infty})}.$$
\end{proof}

\section{Cwikel estimates for magnetic Laplacian}\label{sec_magnetic} 
In this section we apply the abstract Cwikel (see Corollary \ref{abs_Cwikel_main_thm}) estimates for magnetic Laplacian ans briefly explain the connection between the noncommutative Euclidean  space and magnetic Laplacian. 

Let $d\in\bbN$ be even, let $b>0$ and let $\ba$ denotes the magnetic potential 
$$\ba=(a_1,\dots, a_d)=\frac{b}{2}(-M_{t_2},M_{t_1}, -M_{t_4}, M_{t_3},\dots, -M_{t_{d}}, M_{t_{d-1}}),$$ 
which corresponds to the isotropic magnetic field of constant strength $b$.
The magnetic Laplacian is the operator on $L_2(\bbR^d)$ defined as 
$$-\Delta_b=\sum_{i=1}^{d}(\partial_i-a_i)^2,$$
as an essentially self-adjoint operator on the set $C^\infty_0(\bbR^{d})$. 
It is well-known, that the operator $-\Delta_b$ has pure point spectrum consisting of eigenvalues $\{b(2n+\frac{d}2)\}_{n\in\bbZ_+}$ of infinite multiplicity \cite{AHS_magnetic}. Note that the magnetic potential $\ba$ can be written as $\frac{b}2\theta(M_{t_1},\dots, M_{t_d})^T=\frac{b}2\theta(-i\nabla_\theta^T)$, where $\theta$ is the matrix, which defines the noncommutative Euclidean space (see Proposition \ref{nc_plane_to_block_diagonal} with nondegenerate matrix). 

The infinitesimal generator of the group $U(t\be_j)$, $t\in\bbR, j=1,\dots, d$ (where $\{U(\bt)\}_{\bt\in\bbR^d}$ is defined in \eqref{nc_plane_realisation}) is the operator $(\partial_j-a_j)$. Since $\{U(\bt)\}_{\bt\in\bbR^d}$ generates the algebra $L_\infty(\bbR_\theta^d)$, it follows that $(\partial_j-a_j)$ is affiliated with $L_\infty(\bbR_\theta^d)$. Therefore, the operator $-\Delta_b$ is also affiliated with $L_\infty(\bbR_\theta^d)$. Hence, any bounded function $g$ of $-\Delta_b$ is an operator from $L_\infty(\bbR_\theta^d)$. 
At the same time the multiplication operator $M_f$, $f\in L_\infty(\bbR^d)$, is the function of the gradient $-i\nabla_\theta$ associated with $L_\infty(\bbR_\theta^d)$ (see Definition \ref{nc_plane_gradient}). 

Thus, the operator of the form $M_fg(-\Delta_b)$ is of the type which we studied in  Section \ref{sec_nc_plane_Cwikel}, and therefore, one can apply Theorem \ref{nc_plane_Cwikel_p>2} to find estimates on the eigenvalues of the operators of the form $M_fg(-\Delta_b)$. In Theorem \ref{magnetic_Cwikel} below we show a straightforward proof of this result for the special case, when $d=2$. 
%
%
%Noting that $\mu_{L_\infty(\bbR_\theta^d)}(g(-\Delta_b))=\mu_{\ell_\infty(b+2b\bbN, \nu)}(g),$ for any $g\in \ell_\infty(b+2b\bbZ_+, \nu)$, we can reformulate Theorems \ref{nc_plane_Cwikel_p>2} as estimates for the operators of the form $M_fg(-\Delta_b)$. In Theorem \ref{magnetic_Cwikel} below we show a straightforward proof of this result for the special case, when $d=2$. 

Let $n\in\mathbb{Z}_+$ and let $P_n$ be the eigenprojection which corresponds to the eigenvalue $(2n+1)b$. It is well-known (see e.g. \cite{RW}) that the integral kernel of $P_n$ is given by the formula
$$K_n(\bs,\bt)=\frac{b}{2\pi}L_n(\frac{b}{2}|\bs-\bt|^2)\exp(-\frac{b}{4}|\bs-\bt|^2+2i(t_1s_2-t_2s_1)),\quad \bt,\bs\in\mathbb{R}^2.$$
Here, $L_n$ is the $n-$th Laguerre polynomial defined by the formula
$$L_n(u)=\sum_{m=0}^n\binom{n}{m}\frac{(-u)^m}{m!},\quad u\in\mathbb{R}.$$

Firstly, we need an auxiliary lemma.
\begin{lemma}\label{magnetic_projection} For every $n\in\mathbb{Z}_+$ and for every $f\in L_2(\mathbb{R}^2),$ we have $M_fP_n\in \cL_2(L_2(\bbR^2))$ and
$$\|M_fP_n\|_2=(\frac{b}{2\pi})^{\frac12}\|f\|_2.$$
\end{lemma}
\begin{proof} By the definition of a Hilbert-Schmidt norm, we have
\begin{align*}
\|M_fP_n\|_2^2&=\frac{b^2}{4\pi^2}\int_{\mathbb{R}^2}\int_{\mathbb{R}^2}|f|^2(\bs)L_n^2(\frac{b}{2}|\bs-\bt|^2)\exp(-\frac{b}{2}|\bs-\bt|^2)d\bs d\bt\\
&=\frac{b^2}{4\pi^2}\int_{\mathbb{R}^2}|f|^2(\bs)d\bs\cdot\int_{\mathbb{R}^2}L_n^2(\frac{b}{2}|\bt|^2)\exp(-\frac{b}{2}|\bt|^2)d\bt\\
&=\frac{b}{2\pi^2}\int_{\mathbb{R}^2}|f|^2(\bs)d\bs\cdot\int_{\mathbb{R}^2}L_n^2(|\bt|^2)\exp(-|\bt|^2)d\bt\\
&=\frac{b}{2\pi}\int_{\mathbb{R}^2}|f|^2(\bs)d\bs\cdot\int_0^{\infty}L_n^2(u)\exp(-u)du=\frac{b}{2\pi}\int_{\mathbb{R}^2}|f|^2(\bs)d\bs.
\end{align*}
Here, the last equality is guaranteed by \cite[22.2.13]{Abr_Stegun}.
\end{proof}

We define the measure $\nu$ on the set $b+2b\mathbb{Z}_+$ by assigning the weight $2b$ to every atom. We immediately derive the following result, which ensures that the Hypothesis \ref{abs_Cwikel_hyp} is satisfied. One should also compare this lemma with Lemma \ref{nc_plane_l2_cwikel}.

\begin{lemma} For every $f\in L_2(\mathbb{R}^d)$ and for every $g\in L_2(b+2b\mathbb{Z}_+,\nu),$ we have $M_fg(-\Delta_b)\in \cL_2(L_2(\bbR^2))$ and
$$\|M_fg(-\Delta_b)\|_2=\frac1{(2\pi)^{\frac12}}\|f\|_2\|g\|_2.$$
\end{lemma}
\begin{proof}
Since the spectrum of the operator $-\Delta_b$ is $\{b+2bn\}_{n\in\bbZ_+}$ we have that 
$$M_fg(-\Delta_b)=\sum_{n\in\bbZ_+} M_fg(b+2bn)P_n.$$
Therefore, by Lemma \ref{magnetic_projection} we obtain 
\begin{align*}
\|M_fg(-\Delta_b)\|_2^2&=\sum_{n\in\bbZ_+}|g(b+2bn)|^2\|M_fP_n\|_2^2=\|f\|_2^2\sum_{n\in\bbZ_+}\frac{b}{2\pi}|g(b+2bn)|^2\\
&=\frac1{2\pi}\|f\|_2^2\|g\|_2^2.
\end{align*}
\end{proof}
Hence, referring to Theorem \ref{abs_Cwikel_main_thm} we obtain 
\begin{theorem}\label{magnetic_Cwikel}
Let $E(0,\infty)$ be an interpolation space for $(L_2(0,\infty), L_\infty(0,\infty))$. If $f\otimes g\in E( L_\infty(\bbR^2)\otimes \ell_\infty(b+2b\bbZ_+, \nu)),$ then
$M_fg(-\Delta_b)\in E(L_2(\bbR^2))$ and
$$\|M_fg(-\Delta_b)\|_E\leq \const \|f\otimes g\|_{E( L_\infty(\bbR^2)\otimes \ell_\infty(b+2b\bbZ_+, \nu))}.$$ 
\end{theorem}

Note, that in the special case when $b=0,$ and $E=L_p(0,\infty),\, p\geq 2,$ or $E=L_{p,\infty}(0,\infty), p>2$, the result exactly coincides with original Cwikel estimates \cite{Cwikel} (see also \cite{Simon}).

\bigskip

%%%%%%%%%%%%%%%%%%%%%%%%%%%%%%%%%%%%%
\noindent 
{\bf Acknowledgements.} The authors are indebted to their colleagues Alan Carey, Victor Gayral, Steven Lord, Serge Richard, Georgi Raikov, Bruno Iochum   for helpful 
discussions and correspondence concerning various aspects of this paper.

\end{document}